\documentclass[a4paper,oneside,12pt]{article}
\usepackage[titletoc,toc,title]{appendix}

\newcommand{\comw}{\textcolor{white}}

\newcommand{\vta}{\textcolor{violet}}

\newcommand{\homotopy}[3] 
{
	\xymatrix{{#1} \ar@2{~>}@<0.25ex>[r]^{{#2}} & {#3}}
}

\newcommand{\Hom}{\mbox{Hom}}

\newcommand{\Ho}{\cc{H}\mbox{o}(\cc{C},\cc{W})}

\usepackage[latin1]{inputenc}
\usepackage{latexsym}
\usepackage{hyperref}
\usepackage[all]{xy}
\xyoption{rotate}
\usepackage{amsmath,amssymb,amsthm,xspace,a4wide,color}
\usepackage{verbatim} 
\usepackage{xcolor}
\usepackage{enumerate}

\numberwithin{equation}{section}

\newlength\Colsep
\theoremstyle{definition}
  \begingroup
        \newtheorem{remark}[equation]{Remark}
        
        \newtheorem{sinnadastandard}[equation]{\textcolor{white}{-}\hspace{-.3cm}}

\endgroup

\theoremstyle{plain}
  \begingroup
        \newtheorem{theorem}[equation]{Theorem}
        \newtheorem{lemma}[equation]{Lemma}
        \newtheorem{proposition}[equation]{Proposition}
        \newtheorem{corollary}[equation]{Corollary}

        \newtheorem{remarkitalica}[equation]{Remark}
	    \newtheorem{definition}[equation]{Definition}
        
        \newtheorem{sinnadaitalica}[equation]{\textcolor{white}{-}\hspace{-.3cm}}

\endgroup

\newcommand{\cqd}{\hfill$\Box$}  

\makeatletter
\newcommand{\circlearrow}{}
\DeclareRobustCommand{\circlearrow}{%
  \mathrel{\vphantom{\rightarrow}\mathpalette\circle@arrow\relax}%
}
\newcommand{\circle@arrow}[2]{%
  \m@th
  \ooalign{%
    \hidewidth$#1\circ\mkern1mu$\hidewidth\cr
    $#1\longrightarrow$\cr}%
}
\makeatother

\makeatletter
\newcommand{\xRightarrow}[2][]{\ext@arrow 0359\Rightarrowfill@{#1}{#2}}
\newcommand{\xLeftarrow}[2][]{\ext@arrow 0359\Leftarrowfill@{#1}{#2}}
\makeatother

\newcommand{\mr}[1]{\stackrel{#1}{\longrightarrow}}

\newcommand{\Mr}[1]{\stackrel{#1}{\Rightarrow}}
\newcommand{\Ml}[1]{\stackrel{#1}{\Leftarrow}}
\newcommand{\xr}[1]{\xrightarrow{#1}}
\newcommand{\Xr}[1]{\xRightarrow{#1}}
\newcommand{\Xl}[1]{\xLeftarrow{#1}}

\newcommand{\cc}[1]{\mathcal{#1}}
\newcommand{\C}{\mathcal{C}}

\newcommand{\eps}{\varepsilon}

\newcommand{\adj}[2]{
\ar@/^1ex/[r]^{{#1}}
\ar@{}[r]|{\bot}
\ar@/_1ex/@{<-}[r]_{{#2}} }

\newcommand{\jda}[2]{
\ar@/^1ex/[r]^{#1}
\ar@{}[r]|{\top}
\ar@/_1ex/@{<-}[r]_{#2} }

\newcommand{\adjbis}[2]{
\ar@/^2ex/[r]^{{#1}}
\ar@{}[r]|{\bot}
\ar@/_2ex/@{<-}[r]_{{#2}} }

\newcommand{\jdabis}[2]{
\ar@/^2ex/[r]^{#1}
\ar@{}[r]|{\top}
\ar@/_2ex/@{<-}[r]_{#2} }

\newcommand{\idavuelta}[2]{
\ar@/^1ex/[r]^{#1}
\ar@/_1ex/@{<-}[r]_{#2} }

\newcommand{\hpy}
           {
            \hspace{-1ex}
            \xymatrix{ {} \ar@2{~>}@<0.25ex>[r] & {} }
            \hspace{-0.6ex}
            }
\newcommand{\mrhpy}[1]
            {\stackrel{#1}{\hpy}}

\newcommand{\dcell}[1]  
          {
					 \ar@<8pt>@{-}[d]+<-4pt,8pt>
           \ar@<-8pt>@{-}[d]+<4pt,8pt>
           \ar@{}[d]|{#1}
          }

\newcommand{\dcellb}[1]   
          {
           \ar@<10pt>@{-}[d]+<-5pt,8pt>
           \ar@<-10pt>@{-}[d]+<5pt,8pt>
           \ar@{}[d]|{#1}
          }
\newcommand{\dcellbbis}[1]   
{
	\ar@<15pt>@{-}[d]+<-5pt,8pt>
	\ar@<-15pt>@{-}[d]+<5pt,8pt>
	\ar@{}[d]|{#1}
}
\newcommand{\dcellbymedio}[1]   
          {
           \ar@<15pt>@{-}[d]+<-7.5pt,10pt>
           \ar@<-15pt>@{-}[d]+<7.5pt,10pt>
           \ar@{}[d]|{#1}          
          }
\newcommand{\dcellbis}[1]   
          {
           \ar@<15pt>@{-}[d]+<-7.5pt,0pt>
           \ar@<-15pt>@{-}[d]+<7.5pt,0pt>
           \ar@{}[d]|{#1}          
          }          
\newcommand{\dcellbymediobis}[1]   
          {
           \ar@<15pt>@{-}[d]+<-7.5pt,-5pt>
           \ar@<-15pt>@{-}[d]+<7.5pt,0pt>
           \ar@{}[d]|{#1}          
          }
\newcommand{\deq}        
         {
          \ar@{=}[d]
         }

\newcommand{\ddeq}{\ar@{=}[dd]}
\newcommand{\dddeq}{\ar@{=}[ddd]}         
\newcommand{\ddddeq}{\ar@{=}[dddd]}         
\newcommand{\dddddeq}{\ar@{=}[ddddd]}                  
\newcommand{\ddddddeq}{\ar@{=}[dddddd]}                           
\newcommand{\dddddddeq}{\ar@{=}[ddddddd]}         
\newcommand{\ddddddddeq}{\ar@{=}[dddddddd]}         
\newcommand{\dddddddddeq}{\ar@{=}[ddddddddd]}                  
\newcommand{\ddddddddddeq}{\ar@{=}[dddddddddd]}                  
\newcommand{\dddddddddddeq}{\ar@{=}[ddddddddddd]}                  

\newcommand{\dreq}       
         {
          \ar@{=}[dr]
         }

\newcommand{\dleq}       
         {
          \ar@{=}[dl]
         }

\newcommand{\dccell}[1]    
          {
           \ar@{-}[ld]
           \ar@{-}[rd]
           \ar@{}[d]|{#1}
          }

\newcommand{\dcellbb}[1]   
          {
           \ar@<20pt>@{-}[d]+<-10pt,12pt>
           \ar@<-20pt>@{-}[d]+<10pt,12pt>
           \ar@{}[d]|{#1}
          }

\newcommand{\dl}    
          {
           \ar@<-2pt>@{-}[d]+<4pt,8pt>
          }

\newcommand{\dr}    
          {
           \ar@<2pt>@{-}[d]+<-4pt,8pt>
          }
\newcommand{\drbis}    
          {
           \ar@<-2pt>@{-}[d]+<-4pt,8pt>
          }
\newcommand{\drmediobis}    
          {
           \ar@<-1pt>@{-}[d]+<-4pt,8pt>
          }
\newcommand{\dc}[1]    
          {
           \ar@{}[d]|{#1}
          }
\newcommand{\dcr}[1]    
          {
           \ar@{}[dr]|{#1}
          }
\newcommand{\dcl}[1]    
          {
           \ar@{}[dl]|{#1}
          }

\newcommand{\uccell}[1]      
          {
           \ar@{-}[ur]
           \ar@{}[u]|{#1}
           \ar@{-}[ul]
          }

\newcommand{\uccellb}[1]     
          {
           \ar@<-1ex>@{-}[ur]
           \ar@{}[u]|{#1}
           \ar@<1ex>@{-}[ul]
          }

\newcommand{\dcellop}[1]  
          {
					 \ar@<6pt>@{-}[d]+<6pt,8pt>
           \ar@<-6pt>@{-}[d]+<-6pt,8pt>
           \ar@{}[d]|{#1}
          }

\newcommand{\dcellopb}[1]  
          {
					 \ar@<7pt>@{-}[d]+<7pt,8pt>
           \ar@<-7pt>@{-}[d]+<-7pt,8pt>
           \ar@{}[d]|{#1}
          }

\newcommand{\dcellopbb}[1]  
          {
					 \ar@<8pt>@{-}[d]+<8pt,8pt>
           \ar@<-8pt>@{-}[d]+<-8pt,8pt>
           \ar@{}[d]|{#1}
          }
          
\newcommand{\did}{\ar@2{-}[d]}
\newcommand{\ddid}{\ar@2{-}[dd]}

\newcommand{\dig}{ \ar@2{-}[d] & & }

\newcommand{\op}[1]
          {
           \ar@{-}[ld]
           \ar@{-}[rd]
           \ar@{}[d]|{#1}
          }

\newcommand{\opbis}[1]
          {
           \ar@{-}[ld]
           \ar@{-}[rd]
           \ar@{}[d]|>>>>{#1}
          }

\newcommand{\opb}[1]
          {
           \ar@<-2pt>@{-}[ld]
           \ar@<2pt>@{-}[rd]
           \ar@{}[d]|{#1}
          }        
\newcommand{\opmediob}[1]
          {
           \ar@<-1pt>@{-}[ld]
           \ar@<1pt>@{-}[rd]
           \ar@{}[d]|{#1}
          }  
          
\newcommand{\opbymedio}[1]
          {
           \ar@<-3pt>@{-}[ld]
           \ar@<3pt>@{-}[rd]
           \ar@{}[d]|{#1}
          }     
\newcommand{\opbb}[1]
          {
           \ar@<-4pt>@{-}[ld]
           \ar@<4pt>@{-}[rd]
           \ar@{}[d]|{#1}
          }               
\newcommand{\opbbb}[1]
          {
           \ar@<-6pt>@{-}[ld]
           \ar@<6pt>@{-}[rd]
           \ar@{}[d]|>>{#1}
          } 
\newcommand{\opbbbbis}[1]
          {
           \ar@<-6pt>@{-}[ld]
           \ar@<6pt>@{-}[rd]
           \ar@{}[d]|<<<<<{#1}
          } 
\newcommand{\opunodos}[1]
          {
           \ar@{-}[ld]
           \ar@{-}[rrd]
           \ar@{}[dr]|{#1}
          }
          
\newcommand{\opunodosb}[1]
          {
           \ar@<-2pt>@{-}[ld]
           \ar@<2pt>@{-}[rrd]
           \ar@{}[dr]|{#1}
          }

\newcommand{\opdosuno}[1]
          {
           \ar@{-}[lld]
           \ar@{-}[rd]
           \ar@{}[d]|{#1}
          }          
          
\newcommand{\opdosdos}[1]
          {
           \ar@{-}[lld]
           \ar@{-}[rrd]
           \ar@{}[d]|{#1}
          }    
\newcommand{\opdostres}[1]
          {
           \ar@{-}[lld]
           \ar@{-}[rrrd]
           \ar@{}[d]|{#1}
          }    
\newcommand{\optresuno}[1]
          {
           \ar@{-}[llld]
           \ar@{-}[rd]
           \ar@{}[d]|{#1}
          }   
\newcommand{\optresdos}[1]
          {
           \ar@{-}[llld]
           \ar@{-}[rrd]
           \ar@{}[d]|{#1}
          }    

\newcommand{\optrestres}[1]
          {
           \ar@{-}[llld]
           \ar@{-}[rrrd]
           \ar@{}[d]|{#1}
          }            
\newcommand{\opcincocinco}[1]
          {
           \ar@{-}[llllld]
           \ar@{-}[rrrrrd]
           \ar@{}[d]|{#1}
          }

\newcommand{\cl}[1]
          {
           \ar@{-}[ur]
           \ar@{}[u]|{#1}
           \ar@{-}[ul]
          }
\newcommand{\clb}[1]
          {
           \ar@<-1ex>@{-}[ur]
           \ar@{}[u]|{#1}
           \ar@<1ex>@{-}[ul]
          }
\newcommand{\clbb}[1]
          {
           \ar@<-2ex>@{-}[ur]
           \ar@{}[u]|{#1}
           \ar@<2ex>@{-}[ul]
          }
\newcommand{\clmediob}[1]
          {
           \ar@<-.5ex>@{-}[ur]
           \ar@{}[u]|{#1}
           \ar@<.5ex>@{-}[ul]
          } 
\newcommand{\clrightb}[1]
          {
           \ar@<-1ex>@{-}[ur]
           \ar@{}[u]|{#1}
           \ar@{-}[ul]
          }
\newcommand{\clunodos}[1]
          {
           \ar@{-}[urr]
           \ar@{}[u]|{#1}
           \ar@{-}[ul]
          }
          
\newcommand{\cldosuno}[1]
          {
           \ar@{-}[ur]
           \ar@{}[u]|{#1}
           \ar@{-}[ull]
          }
\newcommand{\cldosdos}[1]
          {
           \ar@{-}[urr]
           \ar@{}[u]|{#1}
           \ar@{-}[ull]
          }         
\newcommand{\cltresdos}[1]
          {
           \ar@{-}[urr]
           \ar@{}[u]|{#1}
           \ar@{-}[ulll]
          }           
\newcommand{\cldostres}[1]
          {
           \ar@{-}[urrr]
           \ar@{}[u]|{#1}
           \ar@{-}[ull]
          }     
\newcommand{\cltrestres}[1]
          {
           \ar@{-}[urrr]
           \ar@{}[u]|{#1}
           \ar@{-}[ulll]
          }            
\newcommand{\clcincocinco}[1]
          {
           \ar@{-}[urrrrr]
           \ar@{}[u]|{#1}
           \ar@{-}[ulllll]
          }

\newcommand{\ardr}{\ar@{-}[dr]}
\newcommand{\ardrr}{\ar@{-}[drr]}
\newcommand{\ardrrr}{\ar@{-}[drrr]}
\newcommand{\ardrrrr}{\ar@{-}[drrrr]}
\newcommand{\ardl}{\ar@{-}[dl]}
\newcommand{\ardll}{\ar@{-}[dll]}
\newcommand{\ardlll}{\ar@{-}[dlll]}
\newcommand{\ardllll}{\ar@{-}[dllll]}
\newcommand{\ardlllll}{\ar@{-}[dlllll]}

\newcommand{\cellrdE}[3] 
{\xymatrix@C=7ex@R=2.4ex{
		\ar@<1.6ex>[r]^{#1} 
		\ar@{}@<-1.3ex>[r]^{\!\! {#2} \, \!\Downarrow}
		\ar@<-1.1ex>[r]_{#3} & }
}
\newcommand{\cellrdEcorta}[3] 
{\xymatrix@C=5ex@R=2.4ex{
		\ar@<1.6ex>[r]^{#1} 
		\ar@{}@<-1.3ex>[r]^{\!\! {#2} \, \!\Downarrow}
		\ar@<-1.1ex>[r]_{#3} & }
}
\newcommand{\cellrd}[3] 
 {
  \xymatrix@C=7ex@R=2.4ex
         {
          {} \ar@<1.6ex>[r]^{#1}
             \ar@{}@<-1.3ex>[r]^{\Downarrow \; {#2}}
             \ar@<-1.1ex>[r]_{#3}
          & {}
         }
}
\newcommand{\cellrdb}[3] 
 {
  \xymatrix@C=7ex@R=2.4ex
         {
          {} \ar@<1.9ex>[r]^{#1}
             \ar@{}@<-1.3ex>[r]^{\Downarrow \; {#2}}
             \ar@<-1.1ex>[r]_{#3}
          & {}
         }         
 }
 \newcommand{\scellrd}[3] 
 {
  \xymatrix@C=4.5ex@R=2.4ex
         {
          {} \ar@<1.6ex>[r]^{#1}
             \ar@{}@<-1.3ex>[r]^{\!\! \Downarrow \, {#2}}
             \ar@<-1.1ex>[r]_{#3}
          & {}
         }
}
          
 \newcommand{\modif}[3] 
 {
  \xymatrix@C=7ex@R=2.4ex
         {
          {} \ar@<1.6ex>@{=>}[r]^{#1}
             \ar@{}@<-1.3ex>@{=>}[r]^{\!\! {#2} \, \!\downarrow}
             \ar@{}@<-1.1ex>[r]_{#3}
          & {}
         }
 }

\newcommand{\cellld}[3] 
 {
  \xymatrix@C=6ex@R=2.4ex
         {
            {}
          & {} \ar@<1.0ex>[l]^{#3}
          \ar@{}@<-1.7ex>[l]^{\!\! {#2} \, \!\Downarrow}
	                                 \ar@<-1.7ex>[l]_{#1}
         }
 }

\newcommand{\cellpairrd}[4] 
 {
  \xymatrix@C=8ex@R=2.2ex
         {
          {} \ar@<1.6ex>[r]^{#1}
             \ar@{}@<-1.3ex>[r]^{\!\! \Downarrow \, {#2} 
                                 \;\;\; \Downarrow \, {#3} }
             \ar@<-1.1ex>[r]_{#4}
          & {}
         }
 }

\usepackage[textwidth=0.89in,textsize=scriptsize]{todonotes}

\begin{document}
\title{A localization of bicategories via homotopies}

\author{Descotte M.E., Dubuc E.J., Szyld M.}

\date{\vspace{-5ex}}

\maketitle

\begin{abstract}
Given a bicategory $\C$ and a family $\cc{W}$ of arrows of $\C$, we give conditions on the pair $(\C,\cc{W})$ that allow us to construct the bicategorical localization with respect to $\cc{W}$ by dealing only with the 2-cells, that is without adding objects or arrows to $\C$.
We show that in this case, the 2-cells of the localization can be given by the homotopies with respect to $\cc{W}$, a notion defined in this article which is closely related to Quillen's notion of homotopy for model categories but depends only on a single family of arrows.  
{This localization result has a natural application to the construction of the homotopy bicategory of a model bicategory, which we develop elsewhere, as the pair $(\C_{fc},\cc{W})$ given by the weak equivalences between fibrant-cofibrant objects satisfies the conditions given in the present article.}
\end{abstract}

\section{Introduction} \label{sec:intro}

The subject of this article is the localization of a bicategory $\C$, that is the process of making a family $\cc{W}$ of arrows of $\C$ into equivalences in an appropriate universal sense. 
As far as we know, this situation was first considered in \cite{PRONK2}, where a bicategorical version of the calculus of fractions of \cite{GZ} is given and a localization construction is performed in this case. {This amounts to a pseudofunctor $\C \mr{i} \C[\cc{W}^{-1}]$ which is universal in the following sense: for any bicategory $\cc{D}$, precomposition with $i$, \mbox{$\Hom(\cc{E},\cc{D}) \mr{i^*} \Hom_{\cc{W}}(\C,\cc{D})$} is a biequivalence of bicategories.
Note that here $\Hom_{\cc{W}}(\C,\cc{D})$ stands for a subbicategory of $\Hom(\C,\cc{D})$ considered in \cite{PRONK2}. Its objects are the pseudofunctors that map the arrows of $\cc{W}$ to equivalences, and its arrows are the natural transformations between them that, when interpreted as pseudofunctors from $\C$ into a cylinder category, again map the arrows of $\cc{W}$ to equivalences (see \cite[Th. 21]{PRONK2} for details).}

\vspace{1ex}

As a motivation, let us consider also the example of the homotopy category of a model category \cite{Quillen}.
The homotopy category of a given model category $\C$ is its localization with respect to the weak equivalences, and a construction of it is given in \cite{Quillen} in which the arrows are given by the homotopy classes of arrows of $\C$. 
As is well-known, the localization of a category always exists and can be constructed by {\em adding formal inverses}, that is by identifying classes of zigzags; however this construction is unmanageable in practice. This is a motivation for the constructions in \cite{GZ}, where zigzags of length 2 suffice, and in \cite{Quillen}, where the {\em candidates} for the  inverses are already present in the model category and  the localization can be constructed as a quotient (see also 
\mbox{\cite[10.6]{DHKS}}, or
\mbox{\cite[\S 3.1]{S.paper11}} for a detailed explanation of this situation in an abstract context).

This paper deals with the situation analogous to that of \cite{DHKS}, \cite{S.paper11}, that is the construction of the localization as a quotient, but in dimension 2. This amounts to constructing a localizing bicategory which has the same objects and arrows as the original bicategory. All the difficulty is thus in the 2-cells of this bicategory, which should in a sense include at the same time the original 2-cells of $\C$ and new 2-cells corresponding to a notion of homotopy.
 
Note that if one starts with a category as a trivial bicategory,  the quotient category given by the homotopy relation is 
obtained from our localizing bicategory by applying $\pi_0$ (connected components). 

For an arbitrary bicategory $\C$ and a family $\cc{W}$ of arrows, we 
consider a notion of \mbox{homotopy} between arrows of $\C$, that is a bicategorical notion of homotopy which depends only on the family $\cc{W}$. 
The reader should be aware that homotopies have a direction and are in general non-invertible.  
In fact, the homotopies can be thought of as something that would be an actual $2$-cell if the arrows of $\cc{W}$ were equivalences, and when this is the case 
we can in fact associate to each homotopy $H$ a $2$-cell $\widehat{H}$. We can apply pseudofunctors to homotopies, and thus for any pseudofunctor $\C \mr{F} \cc{D}$ which maps the arrows of $\cc{W}$ to equivalences we can construct in this way a 2-cell $\widehat{FH}$ of $\cc{D}$.
The homotopies are the basic ingredient for the following construction which we introduce in this paper.

\smallskip

\noindent
{\bf The bicategory $\Ho$ and the 2-functor 
$\cc{C} \mr{i} \Ho.$} 
The objects and the arrows of $\Ho$ are those of $\cc{C}$. The 2-cells correspond to equivalence classes of paths of homotopies. Explicitly, a $2$-cell $f \Mr{} g\in\Ho$ is given by the class
$[H^n,\dotsc, H^2,H^1]$ of a finite sequence
$\xymatrix{f \ar@2{~>}@<0.25ex>[r]^{H^1} & f_1 \ar@2{~>}@<0.25ex>[r]^>>>>>>{H^2} & f_2 \dotsb f_{n-1} \ar@2{~>}@<0.25ex>[r]^>>>>>>{H^n} & g}$
 of homotopies, where 
$[H^n,\dotsc,H^{2},H^1] = [K^m,\dotsc,K^{2},K^1]$ if and only if, for every
pseudofunctor $\C \mr{F} \cc{D}$ that maps the arrows of $\cc{W}$ to equivalences, $\widehat{FH^n} \circ \dotsb  \widehat{FH^2} \circ \widehat{FH^1} = \widehat{FK^m} \circ \dotsb  \widehat{FK^2} \circ \widehat{FK^1}.$ 
There is a 2-functor $\C \mr{i} \Ho$, which is the identity on objects and arrows and maps a 2-cell $\mu$ of $\C$ to the class of a homotopy $I^{\mu}$ which satisfies that $\widehat{FI^\mu} = F\mu$ for any $F$ as above (such homotopies exist and are explicitly constructed).

\smallskip

We consider the full subbicategory $\Hom_{(\cc{W},\Theta)}(\C,\cc{D})$ of $\Hom(\C,\cc{D})$ in which the objects are the pseudofunctors that map the arrows of $\cc{W}$ to equivalences (note that, unlike the definition of $\Hom_{\cc{W}}$ given in \cite{PRONK2} and mentioned at the beginning of the present introduction, now the 1-cells are all pseudonatural transformations between these pseudofunctors). 
We prove the following fundamental fact regarding the 2-functor $i$ (Theorem \ref{teo:mediapuposta}, {which holds without any hypothesis on $\cc{W}$}):  precomposing with $i$, 
\mbox{$\Hom(\Ho,\cc{D}) \mr{i^*} \Hom(\C,\cc{D})$} is an isomorphism of bicategories between 
$\Hom_{(i\cc{W},\, \Theta)}(\Ho,\cc{D})$ and 
$\Hom_{(\cc{W},\Theta)}(\C,\cc{D})$. Thus it will be an isomorphism of bicategories between $\Hom(\Ho,\cc{D})$ and 
$\Hom_{(\cc{W},\Theta)}(\C,\cc{D})$ as soon as $i$ maps the arrows of $\cc{W}$ to equivalences.

\vspace{1ex}

Going back to the example of the 
homotopy category of a model category, a reason why in this case the candidates for inverses are present is that, as is well known, any weak equivalence between fibrant-cofibrant objects can be factored as a section followed by a retraction. We say that an arrow is split if it is either a retraction or a section.  Recall that $\cc{W}$ is said to satisfy the ``3 for 2" condition if, for any three arrows that satisfy $fg = h$, whenever two of them are in $\cc{W}$, so is the third. 

We have weakened these notions to formulate their adequate bicategorical versions, and we have shown (Proposition~\ref{prop:otramediapu}): if $\cc{W}$ satisfies 3 for 2, 
then any 
split arrow in $\cc{W}$
is mapped to an equivalence by $i$. 

We define a pseudofunctor $\C \mr{i} \cc{E}$ to be the localization of $\C$ with respect to $\cc{W}$ if it is universal in the following sense: for any bicategory $\cc{D}$, precomposition with $i$, \mbox{$\Hom(\cc{E},\cc{D}) \mr{i^*} \Hom_{(\cc{W},\Theta)}(\C,\cc{D})$} is a biequivalence of bicategories.
We say that this localization is \emph{strict} if these biequivalences are in fact isomorphisms. 
Combining the two results above, that is Theorem \ref{teo:mediapuposta} and Proposition~\ref{prop:otramediapu}, we obtain the main theorem of this article (Theorem~\ref{enfin}): If $\cc{W}$ satisfies 3 for 2 and each arrow of $\cc{W}$ can be written as a composition of split arrows in $\cc{W}$, then $\cc{C} \mr{i} \Ho$ is the strict localization of $\C$ with respect to $\cc{W}$.

{We note that this is a stronger result than what one may expect to get in a bicategorical context, since we get an isomorphism of bicategories and not just an equivalence or a biequivalence.  By the usual reasoning with universal properties, any other construction of the localization of $\C$ with respect to $\cc{W}$ will yield a biequivalent bicategory, not necessarily isomorphic.}


The following is an application we have for this construction of $\Ho$, and the main reason why we have developed it. 
The axioms of model category can be generalized in a natural way to define the notion of {\em model bicategory} (\cite{DDS.model_bicat}, \cite{TesisEmi}, \cite{2-pro}).
The theory of model bicategories has potential applications in the homotopy theory of topoi and in strong shape theory. 
Also, it is expected that this theory could provide a formal setting in which to develop $(\infty,2)$-category theory (see the introduction of \cite{PronkWarren}).
The pair $(\C_{fc},\cc{W})$ given by the weak equivalences between fibrant-cofibrant objects of a model bicategory $\C$ satisfies the hypothesis of 
Theorem~\ref{enfin}, and this allows for an application of this result to the construction of the homotopy bicategory of a model bicategory, which we develop in \cite{DDS.model_bicat}. 

\section{Preliminaries on bicategories}
While the theory of bicategories is nowadays well-established, we prefer to explicitly define its basic concepts in order to fix the notation that we will use throughout the paper.

A bicategory $\cc{C}$ consists of all the following:

\medskip

\noindent 1. A family of objects that we will denote by $X,Y,Z,\dotsc$.

\medskip

\noindent 2. For each pair of objects $X,Y\in\C$ a category $\cc{C}(X,Y)$ whose objects are the arrows $X \mr{f} Y$ of $\cc{C}$ and whose arrows are the 2-cells $\alpha: f \Rightarrow g$ between those arrows. Thus we have a vertical composition of $2$-cells which we denote by ``$\circ$", and identity $2$-cells ``$id_f$". We abuse the notation by denoting indistinctly $f \Mr{id_f} f$ or $f \Mr{f} f$, thus $f = id_f$ as 2-cells. Note that for any 2-cell $\alpha$ as above we have $\alpha \circ f = \alpha = g \circ \alpha$, and in particular $f \circ f = f$.
	
\medskip

\noindent 3. For each $X,Y,Z\in\C$, a functor $\cc{C}(Y,Z) \times \cc{C}(X,Y) \mr{} \cc{C}(X,Z)$. 
	This is a horizontal composition which we denote by ``$\ast\,$", for each configuration $\xymatrix{X \ar@<1.6ex>[r]^{f_1} 
		\ar@{}@<-1.3ex>[r]^{\!\! {\alpha} \, \!\Downarrow}
		\ar@<-1.1ex>[r]_{f_2} & Y      
		\ar@<1.6ex>[r]^{g_1} 
		\ar@{}@<-1.3ex>[r]^{\!\! {\beta} \, \!\Downarrow}
		\ar@<-1.1ex>[r]_{g_2} & Z}$ we have $g_1 \ast f_1 \Xr{\beta \ast \alpha} g_2 \ast f_2$.

\smallskip

All these data have to satisfy the following axioms:

\smallskip

\noindent {\bf H1.} For each $X \mr{f} Y \mr{g} Z\in\C$,  $id_g \ast id_f = id_{g \ast f}$.

\noindent {\bf H2.} For each configuration 
$\xymatrix{X \ar@<3ex>[r]^>>{f_1}_<<<{\alpha \, \Downarrow} 
             \ar[r]^>>{f_2}_<<<{\beta \, \Downarrow} 
             \ar@<-3ex>[r]^>>{f_3} &
 Y
\ar@<3ex>[r]^>>{g_1}_<<<{\gamma \, \Downarrow} 
             \ar[r]^>>{g_2}_<<<{\delta \, \Downarrow} 
             \ar@<-3ex>[r]^>>{g_3} & Z}$, 
$ (\delta \ast \beta) \circ (\gamma \ast \alpha) = 
(\delta \circ \gamma) \ast (\beta \circ \alpha)$. This is the ``Interchange law".

\smallskip     
       
In order to avoid parentheses, we consider ``$\,\ast\,$" more binding than ``$\,\circ\,$", thus 
\mbox{$(\delta \ast \beta) \circ (\gamma \ast \alpha)$} above could be written as $\delta \ast \beta \circ \gamma \ast \alpha$.

\medskip

\noindent 4. Finally, part of the structure of $\C$ is given by the identities, the unitors and the associator as follows:

\smallskip 

\noindent {\bf I.} For each $X\in\C$, we have a 1-cell $X\mr{id_X} X$.

\smallskip 

\noindent {\bf U.} For each $X \mr{f} Y\in\C$, we have invertible 2-cells 
$f \ast id_X \Mr{\lambda} f$, $id_Y \ast f \Mr{\rho} f$.

\smallskip 

\noindent {\bf A.} For each $W \mr{f} X \mr{g} Y \mr{h} Z\in\C$, we have an invertible 2-cell $f \ast (g \ast h) \Mr{\theta} (f \ast g) \ast h$.

\medskip

We will use these same letters $\theta,\rho, \lambda$ for any bicategory, and we will denote the inverses of these 2-cells also by the same letters. 
The unitors and the associators are required to satisfy the 
well-known pentagon and triangle identities (\cite[XII,6]{McL}) and are required to be natural in each of the variables. 
These naturalities are expressed by the following equalities of 2-cells which we record here for convenience: 

\smallskip
\noindent {\bf N$\lambda$.} For each $\xymatrix{X \ar@<1.6ex>[r]^{f} 
             \ar@{}@<-1.3ex>[r]^{\!\! {\alpha} \, \!\Downarrow}
             \ar@<-1.1ex>[r]_{g} & Y}\in\C$, 
$\lambda\circ \alpha \ast id_X =\alpha\circ \lambda $.
          
\smallskip 
\noindent {\bf N$\rho$.} For each $\xymatrix{X \ar@<1.6ex>[r]^{f} 
             \ar@{}@<-1.3ex>[r]^{\!\! {\alpha} \, \!\Downarrow}
             \ar@<-1.1ex>[r]_{g} & Y}\in\C$, 
$\rho \circ id_Y\ast \alpha=\alpha \circ \rho$. 
            
\smallskip              
\noindent {\bf N$\theta$.} For each configuration $\xymatrix{W \ar@<1.6ex>[r]^{f_1} 
             \ar@{}@<-1.3ex>[r]^{\!\! {\alpha} \, \!\Downarrow}
             \ar@<-1.1ex>[r]_{f_2} & X \ar@<1.6ex>[r]^{g_1} 
             \ar@{}@<-1.3ex>[r]^{\!\! {\beta} \, \!\Downarrow}
             \ar@<-1.1ex>[r]_{g_2} & Y \ar@<1.6ex>[r]^{h_1} 
             \ar@{}@<-1.3ex>[r]^{\!\! {\gamma} \, \!\Downarrow}
             \ar@<-1.1ex>[r]_{h_2} & Z}\in\C$, 
$\theta \circ \gamma \ast (\beta \ast \alpha)=
(\gamma  \ast \beta) \ast \alpha \circ \theta$.            

\begin{sinnadastandard} \label{sin:whiskerings}
As is well-known, in order to have a horizontal composition of general 2-cells it is enough to have horizontal compositions between an arrow and a 2-cell: 

Assume that for each 
$\xymatrix{X \ar[r]^{f} & Y 
\ar@<1.6ex>[r]^{g_1} 
             \ar@{}@<-1.3ex>[r]^{\!\! {\alpha} \, \!\Downarrow}
             \ar@<-1.1ex>[r]_{g_2} & Z}$, 
$\xymatrix{X \ar@<1.6ex>[r]^{f_1} 
             \ar@{}@<-1.3ex>[r]^{\!\! {\alpha} \, \!\Downarrow}
             \ar@<-1.1ex>[r]_{f_2} & Y \ar[r]^{g} & Z}\in\C$, 
we have $2$-cells  
$\xymatrix{X 
\ar@<1.6ex>[r]^{g_1 \ast f} 
             \ar@{}@<-1.3ex>[r]^{\!\! {\alpha  \ast  f} \, \!\Downarrow}
             \ar@<-1.1ex>[r]_{g_2 \ast f}             
             & Z}$ 
and
$\xymatrix{X 
\ar@<1.6ex>[r]^{g \ast f_1} 
             \ar@{}@<-1.3ex>[r]^{\!\! {g   \ast  \alpha} \, \!\Downarrow}
             \ar@<-1.1ex>[r]_{g \ast f_2}             
             & Z}$, 
subject to the axioms: 
             
\noindent {\bf W1.} For each 
$\xymatrix{X \ar@<1.6ex>[r]^{f_1} 
             \ar@{}@<-1.3ex>[r]^{\!\! {\alpha} \, \!\Downarrow}
             \ar@<-1.1ex>[r]_{f_2} & Y      
\ar@<1.6ex>[r]^{g_1} 
             \ar@{}@<-1.3ex>[r]^{\!\! {\beta} \, \!\Downarrow}
             \ar@<-1.1ex>[r]_{g_2} & Z}\in\C$,  
$
(g_2 \ast \alpha) \circ (\beta \ast f_1) = (\beta \ast f_2) \circ (g_1 \ast \alpha)
$.
             
\noindent {\bf W2.} For each $X \mr{f} Y \mr{g} Z\in\C$, $\, \, id_g \ast f = g \ast id_f = id_{g \ast f }$.

\noindent {\bf W3.} 
For each 
$\xymatrix{X \ar[r]^{f} & Y
\ar@<2.3ex>[r]^>>{g_1} 
             \ar[r]^>>{g_2}^<<<{\alpha \, \Downarrow}
             \ar@<-2.3ex>[r]^>>{g_3}^<<<{\beta \, \Downarrow} & Z}\in\C$, 
$\, (\beta \ast f)\circ(\alpha \ast f) = (\beta \circ \alpha) \ast f$. 

\hspace{2ex} For each 
$\xymatrix{X \ar@<2.3ex>[r]^<<{f_1}
             \ar[r]^<<{f_2}^>>>{\alpha \, \Downarrow}
             \ar@<-2.3ex>[r]^<<{f_3}^>>>{\beta \, \Downarrow} & Y
             \ar[r]^{g} & Z}\in\C$, 
$ \, (g \ast \beta) \circ (g \ast \alpha) = g \ast (\beta \circ \alpha)$.
            
\vspace{2ex}
            
Then these axioms allow to define, for each configuration as in {\bf W1}, the horizontal composition $\beta \ast \alpha$ by either one of the two compositions there. The correspondence between the sets of axioms ``{\bf H}" and ``{\bf W}" is thus clear. 

\end{sinnadastandard}

We will use this fact in order to define the horizontal composition of a bicategory by defining only the horizontal compositions of 2-cells with arrows. We note also that the axioms {\bf N}$\lambda$ and {\bf N}$\rho$ above involve only these sorts of compositions, and as for axiom {\bf N}$\theta$, it is an easy exercise to show that it is equivalent to the following three axioms, corresponding to putting two identity 2-cells out of the three $\alpha, \beta$ and $\gamma$ in {\bf N}$\theta$:
  
\noindent {\bf N$\theta$1.} For $\xymatrix{W \ar@<1.6ex>[r]^{f_1} 
             \ar@{}@<-1.3ex>[r]^{\!\! {\alpha} \, \!\Downarrow}
             \ar@<-1.1ex>[r]_{f_2} & X \ar[r]^{g} & Y \ar[r]^{h} & Z}\in\C$, $\theta \circ h \ast (g \ast \alpha)=(h \ast g) \ast \alpha \circ \theta$.            
            
\noindent {\bf N$\theta$2.} For $\xymatrix{W \ar[r]^{f} & X \ar@<1.6ex>[r]^{g_1} 
             \ar@{}@<-1.3ex>[r]^{\!\! {\beta} \, \!\Downarrow}
             \ar@<-1.1ex>[r]_{g_2} & Y \ar[r]^{h}& Z}\in\C$, $\theta \circ h \ast (\beta \ast f)=(h \ast \beta) \ast f \circ \theta$.
             
\noindent {\bf N$\theta$3.} For $\xymatrix{W \ar[r]^{f} & X \ar[r]^{g} & Y \ar@<1.6ex>[r]^{h_1} 
             \ar@{}@<-1.3ex>[r]^{\!\! {\gamma} \, \!\Downarrow}
             \ar@<-1.1ex>[r]_{h_2} & Z}\in\C$, $\theta \circ \gamma \ast (g \ast f)=(\gamma \ast g) \ast f \circ \theta$.

\vspace{2ex}

\noindent 
{\bf Coherence.} 
There is a well-known coherence theorem (see for example \cite{basicbicat}) which generalizes the coherence theorem for tensor categories. Given any sequence of composable arrows, the parentheses determine the order in which the compositions are performed. The coherence theorem states that the arrows resulting of any choice of parentheses (and adding or removing identities) are canonically isomorphic by a unique 2-cell built from the associators and the unitors. This justifies the following abuse of notation which greatly simplifies the computations: 

\begin{sinnadastandard} \label{sin:abuse} 
\emph{We write any horizontal composition of arrows omitting the parentheses and the identities. In this way, the associators and the unitors disappear in the diagrams of 2-cells.}
\end{sinnadastandard}

\smallskip

\noindent {\bf Elevator calculus.} 
In addition to the usual pasting diagrams, we will use the Elevator calculus \footnote{Developed in 1969 by the second author for draft use.} to write equations between 2-cells.
In this article, each elevator represents a composition of 2-cells in a bicategory. Objects are omitted, arrows are composed from right to left, and 2-cells from top to bottom. Axiom {\bf H2} shows that the correspondence between elevators and 2-cells is a bijection. Axiom {\bf W1} is the following basic equality for the elevator calculus.
\begin{equation} \label{ascensor}
\vcenter{\xymatrix@C=0ex
         {
             g_1 \dcell{\beta} & f_1 \did
          \\
             g_2 \did & f_1 \dcell{\alpha}
          \\
             g_2  &  f_2 
         }}
\vcenter{\xymatrix@R=6ex{\;\;\;=\;\;\;}}
\vcenter{\xymatrix@C=0ex
         {
             g_1 \did & f_1 \dcell{\alpha}
          \\
             g_1 \dcell{\beta} & f_2 \did
          \\
             g_2 & f_2 
         }}
\vcenter{\xymatrix@R=6ex{\;\;\;=\;\;\;}}
\vcenter{\xymatrix@C=0ex
         {f_2 \dcell{\beta} 
             & f_1 \dcell{\alpha} \\
             g_2 & g_1 
         }}
\end{equation}
This allows to move cells up and down when there are no obstacles, as if they were elevators.

\smallskip

\emph{Using the basic move \eqref{ascensor} we form configurations of cells that fit valid equations in order to prove new equations.}

\begin{definition} \label{def:pseudofunctor}
A pseudofunctor $\C \mr{F} \cc{D}$ between bicategories is given by a \mbox{family} of functors $\C(X,Y) \mr{F} \cc{D}(FX,FY)$, one for each pair of objects $X,Y$ of $\C$, invertible 2-cells $id_{FX} \Xr{\xi_X} F(id_X)$, one for each object $X$ of $\C$ and natural isomorphisms \mbox{$\ast \circ (F\times F) \Mr{\phi} F\circ \ast:\C(X,Y)\times \C(Y,Z) \mr{} \C(X,Z)$} with components \mbox{$Fg \ast Ff \Xr{\phi_{f,g}} F(g \ast f)$,} one for each triplet $X,Y,Z$ of objects of $\C$. As with the associators and unitors, we will omit the subindices of $\xi$ and $\phi$, and use the same letters for the inverses. The following equalities are required to hold:

\vspace{1ex}

For each $X \mr{f} Y\in \C$, \hspace{.2cm}
{\bf P1.} 
$\vcenter{\xymatrix@C=-2ex{Ff \did && \dcell{\xi} \\
Ff && Fid_X  \\
& Ff \cl{\phi} }}
\vcenter{\xymatrix@C=0ex{ \; = \; }}
\vcenter{\xymatrix@C=0ex{Ff \did \\ Ff }} \quad$
\hspace{.2cm}{\bf P2. } 
$\vcenter{\xymatrix@C=-2ex{\dcell{\xi} && Ff \did \\
Fid_X  && Ff \\
& Ff \cl{\phi} }}
\vcenter{\xymatrix@C=0ex{ \; = \; }}
\vcenter{\xymatrix@C=0ex{Ff \did \\ Ff }}$

\vspace{1ex}
 
For each $W \mr{f} X \mr{g} Y \mr{h} Z\in \C$, 
\hspace{.2cm}{\bf P3.} 
$\vcenter{\xymatrix@C=-2.5ex{Fh && Fg {\color{white}XX}  & Ff \did \\
& F(hg) \ar@{}[u]|{\phi} 
\ar@{-}[lu]
\ar@<.5ex>@{-}[ru]
&& Ff \\
&& F(hgf) \cl{\phi}   }}
\vcenter{\xymatrix@C=0ex{ \; = \; }}
\vcenter{\xymatrix@C=-2.5ex{Fh \did & {\color{white}XX} Fg && Ff \\
Fh && F(gf) \ar@{}[u]|{\phi} 
\ar@<-.5ex>@{-}[lu]
\ar@{-}[ru] \\
& F(hgf) \cl{\phi} }}$

\vspace{1ex}

We will often use the naturality of $\phi$, thus we make it explicit: 

For each $\xymatrix{X \ar@<1.6ex>[r]^{f_1} 
             \ar@{}@<-1.3ex>[r]^{\!\! {\alpha} \, \!\Downarrow}
             \ar@<-1.1ex>[r]_{f_2} & Y      \ar@<1.6ex>[r]^{g_1} 
             \ar@{}@<-1.3ex>[r]^{\!\! {\beta} \, \!\Downarrow}
             \ar@<-1.1ex>[r]_{g_2} & Z}\in \C$, \hspace{2ex}              
      {\bf N$\phi$.} 
$\vcenter{\xymatrix@C=-2ex{Fg_1 && Ff_1 \\
& F(g_1 f_1) \cl{\phi} \dcellbb{F(\beta\alpha)} \\
& F(g_2 f_2) }}
\vcenter{\xymatrix@C=0ex{ \; = \; }}
\vcenter{\xymatrix@C=-2ex{Fg_1 \dcellb{F\beta} && Ff_1 \dcellb{F\alpha} \\ 
Fg_2 && Fg_1 \\
& F(g_2 f_2) \cl{\phi} }}$

\vspace{1ex}

A 2-functor is a pseudofunctor such that all the 2-cells $\phi$ and $\xi$ are identities.
\end{definition}

\begin{definition} \ 
\label{pseudo-naturalentrepseudo-functors}
A pseudonatural transformation \mbox{$\theta:{F}\Rightarrow {G}:\cc{C} \rightarrow \cc{D}$} between pseudofunctors consists of a family of arrows 
$FX \xr{\theta_{X}} GX$, one for each $X\in \cc{C}$ 
and a family of invertible 2-cells 
$\vcenter{\xymatrix@C=-1ex{ Gf \dl & \dc{\theta_f} & \dr \theta_X \\
\theta_Y && Ff}}$, one for each $X \mr{f} Y \in \cc{C}$,
satisfying the following axioms:

\smallskip

\noindent {\bf PN0.} For each ${X}\in \C$, 
$\quad \vcenter{\xymatrix@C=0pc
	   {
		   \theta_{{X}} \did 
		& \dcell{\xi}  
		\\
		\theta_{{X}} 
	        &
	        {F}id_{X}
		}}
\vcenter{\xymatrix@C=0pc{\; = \; }}
\vcenter{\xymatrix@C=0pc
       {
         \dcell{\xi} 
	& 
	\theta_{X} \did
	\\
	{G}id_{X} \dl
	& \theta_{X} \ar@{}[dl]|{\theta_{id_{X}}}  \dr
	\\
        \theta_{X}
	&
	{F}id_{X}
	}}$
	
\vspace{1.5ex}

\noindent {\bf PN1.} For each ${X}\stackrel{{f}}\rightarrow {Y}\stackrel{{g}}\rightarrow {Z}\in \C$, 
$\quad \vcenter{\xymatrix@C=-0pc{
 	   {G}{g} \did	 
	   & & 
 	   {G}{f} \dl 
	   &&
 	   \theta_{X} \ar@{}[dll]|{\theta_{f}} \dr
 	   \\
 	   {G}{g} \dl
 	   & & 
 	   \theta_{Y} \ar@{}[dll]|{\theta_{g}} \dr
 	   &&
 	    {F}{f} \did
 	   \\ 
 	   \theta_{{Z}} \did
            & &  {F}{g} &&
            {F}{f} \\
            \theta_{{Z}} 
            &&&
            \!\!\!\!\! {F}{g}{f} \!\!\!\!\! \cl{\phi} &
            }}
 \vcenter{\xymatrix@C=-.4pc{\quad = \quad }}
\vcenter
   {
   \xymatrix@C=-0.5pc
        {{G}{g} 
 		 &&
 		 {G}{f} 
 		 & \quad & \theta_{X} \did 
 	   \\
 	   &
 		 {G}{g}{f} \cl{\phi}   \dl
 		&& &
 		\theta_{X} \ar@{}[dll]_{\theta_{{g}{f}}}  \dr
		\\
 		&
 		 \theta_{Z} 
 		& &&
 		 {F}{g}{f} 
 		}
   }
 $
 
\smallskip

\noindent {\bf PN2.} For each ${X}\cellrd{{f}}{\alpha}{{g}}{Y}\in \C$, $\quad\vcenter{\xymatrix@C=-0pc{
                      {G}{f} \dcellb{{G}\alpha}  
		      &&
		      \theta_{X}  \did 
		      \\
		       {G}{g} \dl
		       &&
		      \theta_{X} \dr \ar@{}[dll]|{\theta_{g}}
		      \\
		       \theta_{{Y}} 
		       && 
		      {F}{g} 
		      }}
     \vcenter{\xymatrix@C=-.4pc{\quad = \quad }}
      \vcenter{\xymatrix@C=-0pc{
		      {G}{f} \dl
		      &&
		      \theta_{X} \dr \ar@{}[dll]|{\theta_{f}} 
		      \\
		      \theta_{{Y}} \did
		      &&
		      {F}{f} \dcellb{{F}\alpha} 
		      \\ 
		      \theta_{Y}
		      &&
		      {F}{g} 
		      }}$
\end{definition}

As a special case, we have the notion of pseudonatural transformation between \mbox{2-functors.} A 2-natural transformation between 2-functors is a pseudonatural transformation such that $\theta_{f}$ is the equality for every arrow ${f}$ of $\cc{C}$.

\begin{definition} 
A \mbox{\emph{modification}} $\rho: \theta \rightarrow \eta: {F}\Rightarrow {G}:\cc{C} \rightarrow \cc{D}$ between 
pseudonatural transformations is a family
 of 2-cells $\theta_{X}\Xr{\rho_{X}}\eta_{X}$ of
$\cc{D}$, one for each $X\in \cc{C}$ such that:

\smallskip

\noindent {\bf PM.} For each $ X\mr{f} Y\in \C,
\; $\hspace{.2cm}
$\vcenter{\xymatrix@C=-0pc{
		      Gf \dl 
		      && 
		      \theta_{X} \dr \ar@{}[dll]|{\theta_f} 
		      \\
		      \theta_Y \dcellb{\rho_Y}
		      && 
		      Ff \did  
		      \\
		      \eta_Y
		      &&
		      Ff
		      }}
     \vcenter{\xymatrix@C=-.4pc{\quad = \quad }}
      \vcenter{\xymatrix@C=-0pc{
		      Gf \did 
		      && 
		      \theta_X \dcellb{\rho_X} 
		      \\
		      Gf \dl 
		      && 
		      \eta_X \dr \ar@{}[dll]|{\eta_f} 
		      \\
		      \eta_Y
		      && 
		      Ff
		      }}$
\end{definition}

Pseudofunctors, pseudonatural transformations and modifications can be composed in order to define, for each pair $\cc{C},\cc{D}$ of bicategories, a bicategory $\Hom(\cc{C},\cc{D})$. We omit the details as they are ubiquitous in the literature.

\begin{definition} \label{def:compositionsubF}
 Let $\cc{C} \mr{F} \cc{D}$ be a pseudofunctor. A configuration $\xymatrix{X \ar@<1ex>[r]^{f_1} 
             \ar@<-1ex>[r]_{f_2} & Y \ar@<1ex>[r]^{g_1} 
             \ar@<-1ex>[r]_{g_2} & Z}\in\cc{C}$ and another one $\xymatrix{FX \ar@<1.6ex>[r]^{Ff_1} 
             \ar@{}@<-1.3ex>[r]^{\!\! {\alpha} \, \!\Downarrow}
             \ar@<-1.1ex>[r]_{Ff_2} & FY \ar@<1.6ex>[r]^{Fg_1} 
             \ar@{}@<-1.3ex>[r]^{\!\! {\beta} \, \!\Downarrow}
             \ar@<-1.1ex>[r]_{Fg_2} & FZ}\in\cc{D}$ determine a $2$-cell 
             $F(g_1 \ast f_1) \Xr{\beta \ast_F \alpha} F(g_2 \ast f_2)$ as the composition $F(g_1 \ast f_1) \Mr{\phi} F(g_1) \ast F(f_1) \Xr{\beta \ast \alpha} F(g_2) \ast F(f_2) \Mr{\phi} F(g_2 \ast f_2)$. Note that if $F$ is a 2-functor, 
$\beta \ast_F \alpha = \beta \ast \alpha$.
\end{definition}

\begin{remarkitalica} \label{rem:circFacellsenImF}
Given a configuration $\xymatrix{X \ar@<1.6ex>[r]^{f_1} 
	\ar@{}@<-1.3ex>[r]^{\!\! {\alpha} \, \!\Downarrow}
	\ar@<-1.1ex>[r]_{f_2} & Y \ar@<1.6ex>[r]^{g_1} 
	\ar@{}@<-1.3ex>[r]^{\!\! {\beta} \, \!\Downarrow}
	\ar@<-1.1ex>[r]_{g_2} & Z}\in\C$, from Definition~\ref{def:compositionsubF} and axiom~{\bf N$\phi$} in Definition \ref{def:pseudofunctor}, it follows that $(F\beta) \ast_F (F\alpha) = F(\beta \ast \alpha)$. 
\end{remarkitalica}

\begin{sinnadastandard}
\label{sin:primerfactorizacion}
 {\bf Factorization of $F$.} 
Let $\cc{C} \mr{F} \cc{D}$ be a pseudofunctor.
We give now a factorization of $F$ which will be very useful later. We define a bicategory $\C_F$, a pseudofunctor $\C_F \mr{F_1} \cc{D}$ and a 2-functor $\C \mr{F_2} \C_F$ such that $F = F_1 F_2$.

 We define the 0 and 1-dimensional aspects of $\C_F$ (that is objects, arrows, identity arrows and horizontal composition of arrows) as the ones of $\C$. {We define the 2-cells of $\C_F$ by putting exactly one 2-cell $f \Mr{\tilde{\alpha}} g$ for each 2-cell $Ff \Mr{\alpha} Fg$ of $\cc{D}$}. Vertical composition of 2-cells is computed in $\cc{D}$, and $id_f$ in $\C_F$ is given by the 2-cell $id_{Ff}$ of $\cc{D}$. 
The composition $\tilde{\beta}\ast \tilde{\alpha}$ in $\C_F$ is given by $\beta \ast_F \alpha$ in Definition~\ref{def:compositionsubF}. 
The axioms {\bf H} follow immediately by the definition of $\,\ast_F\,$ and the corresponding axioms of $\cc{C}$. The unitors and associators of $\C_F$ are obtained by applying $F$ to the ones of $\C$, i.e. they are the 2-cells $\widetilde{F\lambda}$, $\widetilde{F\rho}$, $\widetilde{F\theta}$. Their naturalities and the pentagon and triangle identities all follow in a straightforward way from those of $\cc{C}$, composing when needed with the isomorphism $\phi$. We leave the necessary details to the reader.

The 2-functor $F_2$ is defined by the formulas $F_2X = X$, $F_2f = f$, $F_2\alpha = \widetilde{F\alpha}$. The pseudofunctor $F_1$ is defined by the formulas $F_1X = FX$, $F_1f = Ff$, $F_1\tilde{\alpha} = \alpha$. Its structural 2-cells $\xi$, $\phi$ are given by those of $F$.

This factorization has the following (universal) property: 

\begin{proposition} \label{prop:upoffactoriz}
For any other factorization $\C \mr{G_2} \cc{H} \mr{G_1} \cc{D}$ of $F$ that satisfies that the 0 and 1-dimensional {data} of $\cc{H}$ are those of $\C$ and that $G_2$ is a 2-functor which is the identity on objects and arrows, there is a unique  2-functor $E$ such that $F_2 = E G_2$ and 
$F_1 = E G_1$, as in the following diagram:
$$
\xymatrix
     {
      \cc{C} \ar[rr]^F 
             \ar[rd]^{G_2} 
             \ar@/_2ex/[ddr]_{F_2} 
    &
    & \cc{D}  
    \\
    & \cc{H}    \ar@{-->}[d]^{E}
                \ar[ur]^{G_1} 
    \\ 
    & \cc{C}_F  \ar@/_2ex/[uur]_{F_1}
     }
$$ 

\end{proposition}
\begin{proof}
At the level of objects and arrows, $EX = X$, $Ef = f$ is the only possible definition such that $F_2 = E G_2$.
Now, for each 2-cell 
$f \Mr{\alpha} g$ of {$\cc{H}$} the only possible definition of $E\alpha$ such that 
$F_1 E\alpha = G_1\alpha$ is $E\alpha = \widetilde{G_1 \alpha}$ (recall the definition of $F_1$ on 2-cells). Setting $E\alpha = \widetilde{G_1 \alpha}$, we must check that this determines a 2-functor. It is clear that $E$ preserves vertical compositions if and only if $G_1$ does, and for the horizontal composition we have that for 
$\xymatrix{X \ar@<1.6ex>[r]^{f_1} 
             \ar@{}@<-1.3ex>[r]^{\!\! {\alpha} \, \!\Downarrow}
             \ar@<-1.1ex>[r]_{f_2} & Y \ar@<1.6ex>[r]^{g_1} 
             \ar@{}@<-1.3ex>[r]^{\!\! {\beta} \, \!\Downarrow}
             \ar@<-1.1ex>[r]_{g_2} & Z}$ in $\cc{H}$, 
$E(\beta \ast \alpha) = E\beta \ast E\alpha$ in $\C_F$ if and only if $G_1(\beta \ast \alpha) = G_1\beta \ast_F G_1\alpha = 
\phi \circ (G_1\beta \ast G_1\alpha) \circ \phi$, which is precisely equation {\bf N}$\phi$ for $G_1$ (note that, since $G_2$ is a 2-functor, the structural 2-cells $\phi$ of $G_1$ are exactly the ones of $F$). 
We check now that for every 2-cell $\mu$ of $\C$, $EG_2\mu=F_2\mu$: we have that $F_2 \mu = \widetilde{F \mu} = \widetilde{G_1G_2\mu} = EG_2\mu$. 
Since both $EG_2$ and $F_2$ are 2-functors, we can conclude that they coincide. Also, as we mentioned before, the structural 2-cells of $G_1$ are the ones of $F$ which is also the case for $F_1$. Then, since $E$ is a 2-functor, the structural 2-cells of $F_1E$ and $G_1$ coincide and so we conclude that they are the same pseudofunctor which finishes  the proof.   
\end{proof}

\end{sinnadastandard}

\begin{sinnadastandard} \label{quasiequivalencias}
{\bf Equivalences and quasiequivalences.} 
An arrow $X \mr{f} Y$ of a bicategory is an {\em equivalence} if there exists an arrow $Y \mr{g} X$ (which we call a {\em quasiinverse} of $f$) and invertible 2-cells $g * f \cong id_X$, $f * g \cong id_Y$. It is well-known that these 2-cells can be taken to satisfy the usual triangular identities, and we will assume that this is the case when needed. It is also well-known that $X \mr{f} Y$ is an equivalence if and only if for every object $Z$ we have that the functor $\C(Z,X) \mr{f_*} \C(Z,Y)$ is an equivalence of categories, and if and only if for every $Z$ so is $\C(Y,Z) \mr{f^*} \C(X,Z)$. We denote the family of equivalences of a bicategory with the letter $\Theta$. We say that $f$ is a {\em quasiequivalence} if for every object $Z$ the functors $\C(Z,X) \mr{f_*} \C(Z,Y)$ and $\C(Y,Z) \mr{f^*} \C(X,Z)$ are full and faithful. 
Note that in this case invertible 2-cells are preserved and reflected by these functors. 
We denote the family of quasiequivalences of a bicategory with the letter $\Theta_q$. 


A pseudofunctor $\C \mr{F} \cc{D}$ is a {\em biequivalence of bicategories}
if there exist a pseudofunctor $\cc{D} \mr{G} \C$ (which we call a {\em bi-inverse} of $F$) and pseudonatural transformations $GF \Mr{\alpha} id_{\C}$, $FG \Mr{\beta} id_{\cc{D}}$ which are equivalences {(that is, its components $\alpha_X$, $\beta_X$ are equivalences in $\C$, $\cc{D}$).} {Note that in an equivalence of bicategories $\alpha_X$, $\beta_X$ are isomorphisms, and $G$ is called a {\em quasi-inverse}, while in an isomorphism of bicategories $\alpha_X$, $\beta_X$ are the identity arrows, 
and $G$ is the {\em inverse} of $F$}.
\end{sinnadastandard}

\begin{remark}
{Regarding the factorization $F = F_1 F_2$ of \ref{sin:primerfactorizacion}, for any $X \mr{f} Y$ in $\C$, 
if $Ff$ is a quasiequivalence, so is $F_2f$. We show that $(F_2f)_*$ is full and faithful, the case of $(F_2f)^*$  being dual. Let $g,h: Z \mr{} X$ and $f*g \Mr{\widetilde{\alpha}} f*h$, corresponding to $F(f*g) \Mr{\alpha} F(f*h)$. We have to show that there is a unique 2-cell $g \Mr{} h$ of $\C_F$ whose composition with $f$ equals $\widetilde{\alpha}$.
Consider the composition $\phi \alpha \phi: Ff * Fg \Mr{} Ff * Fh$. Since $Ff$ is a quasiequivalence, there is a unique $Fg \Mr{\beta} Fh$ such that $Ff * \beta = \phi \alpha \phi$, and it is easy to check that $\widetilde{\beta}$ is the unique desired 2-cell.}
\end{remark}

{This remark is false in general for equivalences, which is our reason for considering quasiequivalences in this paper. Since $F_2$ is always a 2-functor, this allows to consider 2-functors instead of arbitrary pseudofunctors in some parts of the paper, which simplifies the computations.}

\section{The homotopy bicategory}
\label{sec:sigmahomot}
We fix a bicategory $\cc{C}$ and a family $\cc{W}$ of arrows of $\C$ containing the identities. We will use the notation $\xymatrix{\cdot \ar[r]|{\circ} & \cdot}$ for the arrows of $\cc{W}$. We develop a theory of homotopies and cylinders with respect to the class $\cc{W}$ (instead of working with three distinguished classes as is the case for model categories).
The main result we show is that the homotopies form the 2-cells of a bicategory which, under two natural hypotheses on $\cc{W}$, is the localization of $\C$ with respect to $\cc{W}$, in the sense that it universally turns these arrows into equivalences. 
\begin{definition} \label{defcylestiloQuillen}
Let $X \in \C$. A cylinder $C$ (for $X$, with respect to $\cc{W}$) is given by the data $C = (W,Z,d_0,d_1,x,s,\alpha_0,\alpha_1)$, 
fitting in $\vcenter{\xymatrix@C=1.5pc@R=1.5pc{X \ar[rr]^{d_0} \ar[dd]_{d_1} \ar[rrdd]|{\comw{M^M} x \comw{M^M}} & \ar@{}[dr]|{\cong \; \Downarrow \; \alpha_0 \;} & W \ar[dd]|{\circ}^{s}\\
                   \ar@{}[dr]|{ \; \cong \; \Uparrow \; \alpha_1} && \\ 
                  W \ar[rr]|{\circ}_{s}  && Z}}$. We denote the invertible 2-cell $s\ast d_0 \Xr{\alpha_0} x \Xr{\alpha_1^{-1}} s\ast d_1$ by $\alpha=\alpha_1^{-1}\circ \alpha_0$.
\end{definition}

For comparison with Quillen's definition of cylinder object in \cite{Quillen}, assume that $\C$ is a model category that we regard as a discrete bicategory, and $\cc{W}$ is the family of weak equivalences. Then a cylinder object as in \cite[I.1, Def. 4]{Quillen} is a cylinder as in Definition \ref{defcylestiloQuillen} that satisfies some extra conditions, in particular $Z=X$, $x=id_X$.
    

Consider now any cylinder (either as in the present paper or a cylinder object as in \cite{Quillen}) $C$ for an object $X$, and any arrow $X' \mr{\ell} X$. Then we can define a cylinder \mbox{$C' = (W,Z,d_0*\ell,d_1*\ell,x*\ell,s,\alpha_0*\ell,\alpha_1*\ell)$} for $X'$, which we think of as the composition $C * \ell$.
Now, if $C$ is a classical cylinder object as in \cite{Quillen}, this is not anymore the case for $C'$ (note that this provides in particular new examples of cylinders as defined here). In this way this new notion of cylinder arises as a natural way to define the composition of arrows with homotopies (see \ref{homotopy*arrow} below for details). This is not only convenient for the computation of such compositions, but also turns out to be unavoidable in order to construct the homotopy bicategory with respect to a single class of weak equivalences instead of a full model structure.

\begin{definition} \label{def:Cilindroinversaeidentidad}
Given a cylinder $C$ as above, 
we define the \emph{inverse cylinder}  \mbox{$C^{-1} = (W,Z,d_1,d_0,x,s,\alpha_1,\alpha_0)$.}

\noindent
Also, for $X \in \C$ we define an \emph{identity cylinder} \mbox{$C_X = (X,X,id_X,id_X,id_X,id_X,id_X,id_X)$} (recall our abuse of notation \ref{sin:abuse}).
\end{definition}

\begin{definition}  \label{defhpy}
Let $f,g: X \rightarrow Y\in \C$. A \emph{left homotopy} (with respect to $\cc{W}$) $H$ from $f$ to 
$g$, which we will denote by $f \mrhpy{H} g$, is given by the data  
$H = (C,h,\eta,\eps)$, where $C$ is a cylinder for $X$ as in Definition \ref{defcylestiloQuillen}, $h$ is an arrow $W \mr{h} Y$ and $\eta,\eps$ are 2-cells $f \Mr{\eta} h * d_0$, $h * d_1 \Mr{\eps} g$.
We organize the data of a homotopy as follows
\begin{equation} \label{eq:H}
\begin{tabular}{cc}
$f \mrhpy{H} g$: \hspace{3ex}  
$\vcenter{\xymatrix{X \ar@<-.5ex>[dr]_{x} \ar@<1ex>[r]^{d_0} 
             \ar@<-1ex>[r]_{d_1} & W \ar[r]^h \ar[d]|{\circ}^s & Y \\ & Z}}$ &
\begin{tabular}{c}
$f \Mr{\eta} h \ast d_0$ \\
$s \ast d_0 \Mr{\alpha_0} x \Ml{\alpha_1} s \ast d_1$ \\
$h \ast d_1 \Mr{\eps} g$ 
\end{tabular}
\end{tabular}
\end{equation}
We say that $H$ has invertible cells if $\eta$ and $\eps$ are invertible (recall that $\alpha_0$ and $\alpha_1$ are always required to be invertible). 
\end{definition}

Throughout this article we will work only with left homotopies, and thus omit to write the word ``left". 
As in dimension 1, left homotopies suffice to construct the localization and right homotopies (which  correspond to considering the same class $\cc{W}$ in $\C^{op}$) are only needed for the construction of the homotopy bicategory of a model bicategory that we develop elsewhere.

\begin{definition} \label{def:Halamenos1}
If $H$ as in Definition~\ref{defhpy} has invertible cells, we define a homotopy $H^{-1} = (C^{-1},h,\eps^{-1},\eta^{-1})$ from $g$ to $f$.
\end{definition}

\begin{definition} \label{def:construccionK}
Any cylinder $C$ as in Definition~\ref{defcylestiloQuillen} determines a homotopy 
\mbox{$d_0 \mrhpy{H^C} d_1$,} (recall our abuse of notation \ref{sin:abuse}):
\begin{center}
\begin{tabular}{cc}
$d_0 \mrhpy{H^{C}} d_1$: \hspace{3ex}  
$\vcenter{\xymatrix{X \ar[rd]_x \ar@<1ex>[r]^{d_0} 
            \ar@<-1ex>[r]_{d_1} & W \ar[r]^{id_W} \ar[d]|{\circ}^s & W \\ & Z}}$ &
\begin{tabular}{c}
$d_0 \Xr{d_0} d_0$ \\
$s \ast d_0 \Mr{\alpha_0} x \Ml{\alpha_1} s \ast d_1$ \\
$d_1 \Xr{d_1} d_1$ 
\end{tabular}
\end{tabular}
\end{center}             
\end{definition}   

We now make various constructions for homotopies that we will use later. In these definitions we omit parentheses according to the abuse of notation~\ref{sin:abuse}. Let $H$ be as in \eqref{eq:H}:

\begin{sinnadastandard} \label{2-cellohomotopy}
If $g \Mr{\mu} g'\in\C$, we define a homotopy 
$\mu \circ H$ from $f$ to  $g'$ as follows 
\begin{center}
\begin{tabular}{cc}
$f \mrhpy{\mu \circ H} g'$: \hspace{2ex}
$\vcenter{\xymatrix{X \ar@<-.5ex>[dr]_{x} \ar@<1ex>[r]^{d_0} 
             \ar@<-1ex>[r]_{d_1} & W \ar[r]^{h} \ar[d]|{\circ}^s & Y' \\ & Z}}$ 
\begin{tabular}{rcl}
& $f \Mr{\eta} h \ast d_0$ \\
& $s \ast d_0 \Mr{\alpha_0} x \Ml{\alpha_1} s \ast d_1$  \\
& $h \ast d_1 \Mr{\varepsilon} g \Mr{\mu} g'$
\end{tabular}
\end{tabular}
\end{center}
\end{sinnadastandard}

\begin{sinnadastandard} \label{homotopyo2-cell}
If $f' \Mr{\nu} f\in\C$, we define a homotopy 
$H \circ\nu$ from $f'$ to  $g$ as follows 
\begin{center}
\begin{tabular}{cc}
$f' \mrhpy{H \circ\nu} g$: \hspace{2ex}
$\vcenter{\xymatrix{X \ar@<-.5ex>[dr]_{x} \ar@<1ex>[r]^{d_0} 
             \ar@<-1ex>[r]_{d_1} & W \ar[r]^{h} \ar[d]|{\circ}^s & Y' \\ & Z}}$ 
\begin{tabular}{rcl}
& $f'  \Mr{\nu} f \Mr{\eta} s \ast d_0$ \\
& $s \ast d_0 \Mr{\alpha_0} x \Ml{\alpha_1} s \ast d_1$  \\
& $h \ast d_1 \Mr{\varepsilon} g$
\end{tabular}
\end{tabular}
\end{center}
\end{sinnadastandard}

\begin{sinnadastandard} \label{arrow*homotopy}
If $Y \mr{r} Y'\in\C$, we define a homotopy $r \ast H$ from $r \ast f$ to $r \ast g$ as follows 
\begin{center}
\begin{tabular}{cc}
$r \ast f \mrhpy{r \ast H} r \ast g$: \hspace{2ex}
$\vcenter{\xymatrix{X \ar@<-.5ex>[dr]_{x} \ar@<1ex>[r]^{d_0} 
             \ar@<-1ex>[r]_{d_1} & W \ar[r]^{r \ast h} \ar[d]|{\circ}^s & Y' \\ & Z}}$ 
\begin{tabular}{c}
$r \ast f \Xr{r \ast \eta} r \ast h \ast d_0$ \\
$s \ast d_0 \Mr{\alpha_0} x \Ml{\alpha_1} s \ast d_1$  \\
$r \ast h \ast d_1 \Xr{r \ast \eps} r \ast g $
\end{tabular}
\end{tabular}
\end{center}
\end{sinnadastandard}

\begin{sinnadastandard} \label{homotopy*arrow}
If $X' \mr{\ell} X\in\C$, we define a homotopy $H \ast \ell$ from $f \ast \ell$ to $g \ast \ell$ as follows 
\begin{center}
\begin{tabular}{cc}
$f \ast \ell \mrhpy{H \ast \ell} g \ast \ell$: \hspace{2ex}
$\vcenter{\xymatrix{X' \ar@<-.5ex>[dr]_{x \ast \ell} \ar@<1ex>[r]^{d_0 \ast \ell} 
             \ar@<-1ex>[r]_{d_1 \ast \ell} & W \ar[r]^h \ar[d]|{\circ}^s & Y \\ & Z}}$ &
\begin{tabular}{c}
$f \ast \ell \Xr{\eta \ast \ell} h \ast d_0 \ast \ell$ \\
$s \ast d_0 \ast \ell \Xr{\alpha_0 \ast \ell} x \ast \ell \Xl{\alpha_1 \ast \ell} s \ast d_1 \ast \ell$ \\
$h \ast d_1 \ast \ell \Xr{\eps \ast \ell} g \ast \ell$
\end{tabular}
\end{tabular}
\end{center}
\end{sinnadastandard}

\begin{remark} \label{rem:Hcomo3sinclases}
For any homotopy $H$ as in \eqref{eq:H}, we have $H = \eps \circ (h * H^C) \circ \eta$.
\end{remark}

\begin{definition} \label{def:muinduceH}
A $2$-cell $\xymatrix{X \ar@<1.6ex>[r]^{f} 
             \ar@{}@<-1.3ex>[r]^{\!\! {\mu} \, \!\Downarrow}
             \ar@<-1.1ex>[r]_{g} & Y}\in\C$ yields two homotopies $H_0^\mu, H_1^\mu: f \mrhpy{} g$,
$H_0^\mu = (C_X,g,\mu,g)$ and
$H_1^\mu = (C_X,f,f,\mu)$:

\begin{center}             
\begin{tabular}{lc}
$H^{\mu}_0: \vcenter{\xymatrix{X \ar@<-.5ex>[dr]_{id_X} \ar@<1ex>[r]^{id_X} 
             \ar@<-1ex>[r]_{id_X} & X \ar[r]^g \ar[d]|{\circ}^{id_X} & Y \\ & X}} \quad\quad$ & 
             $H^{\mu}_1: \vcenter{\xymatrix{X \ar@<-.5ex>[dr]_{id_X} \ar@<1ex>[r]^{id_X} 
             \ar@<-1ex>[r]_{id_X} & X \ar[r]^f \ar[d]|{\circ}^{id_X} & Y \\ & X}} \quad\quad$ \\
$\eta = \mu$, $\alpha_0 = \alpha_1 = id_X$, $\eps = g \quad\quad$   
&
$\quad\quad \eta = f$, $\alpha_0 = \alpha_1 = id_X$, $\eps = \mu$
\end{tabular}
\end{center}
\end{definition}

The homotopies can be thought of as something that would be an actual $2$-cell if the arrows of $\cc{W}$ were equivalences (more generally if they were quasiequivalences, recall $\ref{quasiequivalencias})$. When this is the case, cylinders and homotopies yield actual $2$-cells of $\C$ as follows:

\begin{definition} \label{beta} Consider a cylinder $C$ as in Definition~\ref{defcylestiloQuillen}, with $s$ a quasiequivalence.

\smallskip

\noindent
1. We denote by $d_0 \Mr{\widehat{C}} d_1$ the unique invertible $2$-cell such that $s \ast \widehat{C} = \alpha$. 
\smallskip

\noindent
2. For a homotopy $H$ with cylinder $C$, we note by $\widehat{H}$ the  2-cell \mbox{$f \Mr{\eta} h \ast d_0 \Xr{h \ast \widehat{C}} h \ast d_1 \Mr{\eps} g$.}

Note that we have $\widehat{H^C} = \widehat{C}$. Item 2 in this definition can be considered as the extension of this formula to an arbitrary $H$ using Remark~\ref{rem:Hcomo3sinclases}.
\end{definition}

Consider now another family $\Gamma$ of arrows of a bicategory $\cc{D}$. For pseudofunctors \mbox{$\C \mr{F} \cc{D}$,} we write \mbox{$(\C,\cc{W}) \mr{F} (\cc{D},\Gamma)$} to denote that $F$ maps the arrows of $\cc{W}$ to $\Gamma$. We can apply the pseudofunctor $F$ to cylinders and homotopies of $\cc{C}$ as follows:

\begin{definition}
Let $(\C,\cc{W}) \mr{F} (\cc{D},\Gamma)$.

\smallskip

\noindent 1. For a cylinder $C$ as in Definition \ref{defcylestiloQuillen}, we define the cylinder $FC$ by $$FC = (FW,FZ,Fd,Fc,Fx,Fs,F\alpha_0 \circ \phi,F\alpha_1 \circ \phi).$$

\smallskip

\noindent 2. For a homotopy $H$ as in Definition \ref{defhpy}, we define the homotopy $\xymatrix{ Ff \ar@2{~>}@<0.25ex>[r]^{FH} & Fg}$ by $$FH = (FC,Fh,\phi \circ F\eta,F\eps \circ \phi).$$

The constructions of $FC$ and $FH$ are more clearly understood using the diagram

\begin{center}
\begin{tabular}{c} 
$\vcenter{\xymatrix{FX \ar@<-.5ex>[dr]_{Fx}  \ar@<1ex>[r]^{Fd_0} 
             \ar@<-1ex>[r]_{Fd_1} & FC \ar[r]^{Fh} \ar[d]|{\circ}^{Fs} & FY \\ & FZ}}$
\begin{tabular}{c}
$Ff \Xr{F\eta} F(h \ast d_0) \Mr{\phi} Fh \ast Fd_0$ \\
$Fs \ast Fd_0 \Mr{\phi} F(s \ast d_0) \Xr{F\alpha_0} Fx \Xl{F\alpha_1} F(s \ast d_1) \Ml{\phi} Fs \ast Fd_1$ \\
$Fh \ast Fd_1 \Mr{\phi} F(h \ast d_1) \Xr{F\eps} Fg$ 
\end{tabular}
\end{tabular}
\end{center}

\end{definition}

\begin{definition} \label{def:primerdefadhoc}
Recall that $\Theta_q$ denotes the class of quasiequivalences. 
We identify two homotopies $H$, $K$ if for every pseudofunctor $(\C,\cc{W}) \mr{F} (\cc{D},\Theta_q)$, $FH$ and $FK$ yield the same 2-cell (as in Definition  \ref{beta}) of $\cc{D}$, that is $[H] = [K] \iff \widehat{FH} = \widehat{FK}$ for every pseudofunctor $(\C,\cc{W}) \mr{F} (\cc{D},\Theta_q)$. 
\end{definition}

We will see below that it suffices to require the condition in Definition~\ref{def:primerdefadhoc} only for 2-functors $F$. The 2-cell $\widehat{FH}$ is the composition
\begin{equation}\label{eq:FHsombrero}
\widehat{FH}:
Ff \Mr{F\eta} F(h \ast d_0) \Mr{\phi} Fh \ast Fd_0 \Xr{Fh \ast \widehat{FC}} Fh \ast Fd_1 \Mr{\phi} F(h \ast d_1) \Mr{F\eps} Fg,
\end{equation}
where $Fd_0 \Xr{\widehat{FC}} Fd_1$ is the unique $2$-cell such that $Fs \ast \widehat{FC} = \phi \circ F\alpha \circ \phi$. 
With the notation of Definition~\ref{def:compositionsubF}, this can be stated as:

\begin{remark} \label{def:FyHinducen2cell}
 For a homotopy $H$ as in Definition~\ref {defhpy}, and a pseudofunctor 
 \mbox{$(\C,\cc{W}) \mr{F} (\cc{D},\Theta_q)$,} $\widehat{FH}$ is the composition 
$Ff \Mr{F\eta} F(h \ast d_0) \Xr{Fh \ast_F \widehat{FC}} F(h \ast d_1) \Mr{F\eps} Fg$, where $\widehat{FC}$ is the unique $2$-cell such that $Fs \ast_F \widehat{FC} = F\alpha$. Note that when $F$ is a 2-functor, $\, \ast_F = \ast \,$.
\end{remark}

\begin{remark} \label{rem:soloimportaalphatilde}
It is the composition $\alpha = \alpha_1^{-1} \circ \alpha_0$ which is used in order to determine the class of a homotopy. 
This suggests that we can also define a notion of 
cylinder in which $\alpha_0$ and $\alpha_1$ are replaced by a single (invertible) 2-cell $s * d_0 \Mr{\alpha} s * d_1$. Note that all the constructions of this paper work also for the corresponding notion of homotopy. 
In fact, under the hypothesis of Theorem \ref{enfin}, the two resulting homotopy bicategories are the same. This can be either seen by a direct computation, or showing (with the same proof) that both constructions satisfy that Theorem. 
Our reason for considering the notion in Definition \ref{defcylestiloQuillen}, even though it involves more data than the one with a single 2-cell, comes from our interpretation of cylinders as a bicategorical version of a {cylinder} in dimension 1, in which identities are replaced by invertible 2-cells. In a cylinder, $d_0$ and $d_1$ are interchangeable, and we consider this symmetry is better reflected in the definition with two 2-cells than with {a $2$-dimensional version of a commuting square, that is, a single invertible 2-cell} for which one is forced to choose a direction.

%
%
%
\end{remark}

Consider a pseudofunctor $(\C,\cc{W}) \mr{F} (\cc{D},\Theta_q)$, and consider the factorization of $F$ through $\C_F$ given in \ref{sin:primerfactorizacion}. Note that, as  was explained in \ref{quasiequivalencias}, we have 
\mbox{$(\C,\cc{W}) \mr{F_2} (\cc{C}_F,\Theta_q)$.}
The following remark follows immediately by considering Remark~\ref{def:FyHinducen2cell} for $F$ and for $F_2$ (recall that the horizontal composition of 2-cells in 
$\cc{C}_F$ is given by $\, \ast_F \,$).

\begin{remarkitalica} \label{prop:bastaconlos2functors}
Consider the situation in Definition \ref{def:compositionsubF}. We have:

\begin{enumerate}
	\item\label{it:F1F2} $F_1(\widehat{F_2H}) = {\widehat{FH}}$.
	
	\item\label{it:HK} Let $H$, $K$ be any two homotopies, then 
	$$
	\widehat{FH} = \widehat{FK} \;\text{for every pseudofunctor} \; F 
	\iff  \widehat{FH} = \widehat{FK} \; \text{for every 2-functor} \; F. 
	$$   
	(where $F$ labels an arrow  $(\C,\cc{W}) \mr{F} (\cc{D},\Theta_q)$).
\end{enumerate}
\end{remarkitalica}
\begin{proof}
Note that $F_1(\beta \ast_{F_2} \alpha) = \beta \ast_{F} \alpha$, \ref{it:F1F2} follows from this equality, \ref{it:HK} follows from \ref{it:F1F2}.
\end{proof}

The previous remark allows us to consider 2-functors $(\C,\cc{W}) \mr{F} (\cc{D},\Theta_q)$ instead of arbitrary pseudofunctors in Definition~\ref{def:primerdefadhoc}.

\begin{proposition} \label{prop:muinduceH}
Consider the homotopies of Definition~\ref{def:muinduceH}. Then, for any 2-functor  
$(\C,\cc{W}) \mr{F} (\cc{D}, \, \Theta_q)$, $\widehat{FH_0^{\mu}} = \widehat{FH_1^{\mu}} = F\mu$. 
\end{proposition}
\begin{proof}
In the notation of Definition~\ref {defhpy} the 
homotopy $H_0^{\mu}$ has 
$s = id_X$, $\eta = \mu$, $\alpha = id_{id_X}$, $\varepsilon = id_g$, thus $\widehat{FC} = id_{id_X}$, and $\widehat{FH_0^{\mu}}$ is the composition $Fid_g \circ Fid_{id_X} \circ F\mu = F\mu$.
The case of $H_1^{\mu}$ is similar.
\end{proof}

\begin{definition} \label{Imu}
Given any 2-cell $\mu\in\C$, the notation $I^\mu$ stands for any homotopy such that for any 2-functor $(\C,\cc{W}) \mr{F} (\cc{D}, \, \Theta_q)$, $\widehat{FI^\mu} = F\mu$; note that in view of the previous proposition such a $I^\mu$ always exists.
\end{definition}  

\begin{proposition} \label{prop:cuadradowhiskering}
Let $(\C,\cc{W}) \mr{F} (\cc{D},\Theta_q)$ be a pseudofunctor and let $H$ be as in \eqref {eq:H}. Then:
\begin{enumerate}
\item For each $g \Mr{\mu} g'\in\C$ as in \ref{2-cellohomotopy}, we have $\widehat{F(\mu \circ H)}= F\mu \circ \widehat{FH}$.
\item For each $f' \Mr{\nu} f\in\C$ as in \ref{homotopyo2-cell}, we have $\widehat{F(H \circ \nu)} = \widehat{FH} \circ F\nu$.
\item For each $Y \mr{r} Y'\in\C$ as in \ref{arrow*homotopy} we have $\widehat{F(r \ast H)} = Fr \ast_F \widehat{FH}$.
\item For each $X' \mr{\ell} X\in\C$ as in \ref{homotopy*arrow} we have $\widehat{F(H \ast \ell)} = \widehat{FH} \ast_F F\ell$.
\end{enumerate}
\end{proposition}

\begin{proof}  Items 1 and 2 are immediate. 
We show first items 3 and 4 assuming that $F$ is a 2-functor (recall that in this case $\ast_F$ is just $\ast$). Let $Fd \Mr{\widehat{FC}} Fc$ be the unique 2-cell such that $Fs \ast \widehat{FC} = F\alpha$.

\vspace{1ex} 

\noindent Proof of 3: $\widehat{F(r \ast H)}$ is the 2-cell 
$$Fr \ast Ff \Xr{Fr \ast \eta} Fr \ast Fh \ast Fd_0 \Xr{Fr \ast Fh \ast \widehat{FC}} 
Fr \ast Fh \ast Fd_1 \Xr{Fr \ast \eps} Fr \ast Fg,$$

\noindent $\quad \quad \quad \quad \quad$ which is equal to $Fr \ast \widehat{FH}$. 

\smallskip

\noindent Proof of 4: We have $Fs \ast \widehat{FC} \ast F\ell = \alpha \ast F\ell$, and thus 
$\widehat{F(H \ast \ell)}$ is the 2-cell
$$Ff \ast F\ell \Xr{\eta \ast F\ell} Fh \ast Fd_0 \ast F\ell \Xr{Fh \ast \widehat{FC} \ast F\ell} Fh \ast Fd_1 \ast F\ell \Xr{\eps \ast F\ell} Fg \ast F\ell,$$

\noindent $\quad \quad \quad \quad \quad$ which is equal to $\widehat{FH} \ast F\ell$.

\smallskip

\noindent If $F$ is a pseudofunctor, we have
$$\widehat{F(r \ast H)} = F_1 \widehat{F_2(r \ast H)} = F_1 (F_2r \ast \widehat{F_2H})= F_1 (r \ast \widehat{F_2H})=Fr \ast_F \widehat{FH},$$

\noindent where the first equality holds by Remark~\ref{prop:bastaconlos2functors} and the last one is due to Remark~\ref{rem:circFacellsenImF} plus the fact that the structural cells of $F_1$ are those of $F$. The case of item 4 is dual.
\end{proof}

\bigskip

\noindent {\bf The bicategory $\Ho$ and the 2-functor 
$\cc{C} \mr{i} \Ho.$}
We will construct here the bicategory $\Ho$, whose 2-cells are given by the homotopies.
In this more general case, it seems that two arbitrary homotopies can't be vertically composed {\em \`a la}  Quillen (see \cite[Lemma 3]{Quillen}).
This is why finite sequences of composable homotopies have to be considered instead
\footnote{We deal in Appendix \ref{sec:vertcomp} with a set of hypothesis that allow to compose homotopies so that single homotopies can be used instead of finite sequences.}.
%
 We begin by extending Definition~\ref{def:primerdefadhoc} to these sequences as follows:

\begin{definition}  \label{eq:equiventrehomot}
Two finite sequences of homotopies $\xymatrix{f \ar@2{~>}@<0.25ex>[r]^{H^1} & f_1 \ar@2{~>}@<0.25ex>[r]^>>>>>>{H^2} & f_2 \dotsb f_{n-1} \ar@2{~>}@<0.25ex>[r]^>>>>>>{H^n} & g}$,
 $\xymatrix{f \ar@2{~>}@<0.25ex>[r]^{K^1} & f'_1 \ar@2{~>}@<0.25ex>[r]^>>>>>>{K^2} & f'_2 \dotsb  f'_{m-1} \ar@2{~>}@<0.25ex>[r]^>>>>>>{K^m} & g}$ are considered equivalent by the following definition:  

\vspace{1ex}

\noindent
$  [H^n,\dotsc,H^{2},H^1] = [K^m,\dotsc,K^{2},K^1]  \iff $ 

for every 2-functor $(\C,\cc{W}) \mr{F} (\cc{D},\Theta_q)$, $\widehat{FH^n} \circ \dotsb  \widehat{FH^2} \circ \widehat{FH^1} = \widehat{FK^m} \circ \dotsb  \widehat{FK^2} \circ \widehat{FK^1}.$
 \end{definition}
 
\begin{remarkitalica} \label{rem:bastaconlos2functors}
Note that, by Remark~\ref{prop:bastaconlos2functors}, it is equivalent  
to state the condition above for every pseudofunctor.
\end{remarkitalica}

\noindent We construct now a bicategory which we refer to as the homotopy bicategory of $\C$ with respect to $\cc{W}$ and denote by $\Ho$: 

\begin{sinnadaitalica} \label{def:2cellsdeHoC} \label{bicategoryHo}
The objects and the arrows of $\Ho$ are again the objects and arrows of $\cc{C}$. The 2-cells of $\Ho$ are, loosely speaking, the homotopies of $\cc{C}$. More precisely, a $2$-cell $f \Mr{} g\in\Ho$ is given by the class $[H^n,\dotsc, H^2,H^1]$ of a finite sequence of homotopies.
\end{sinnadaitalica}

\begin{remarkitalica} 
Note that by Definition~\ref{Imu} all possible homotopies $I^\mu$  determine the same class in $\Ho$. In particular by  Proposition~\ref{prop:muinduceH} this is the case for the two \mbox{homotopies} $H^{\mu}_0$ and $H^{\mu}_1$ in Definition~\ref{def:muinduceH}.
\end{remarkitalica}

\smallskip

\noindent
{{\bf Vertical composition.}} Vertical composition is defined by juxtaposition:
\begin{sinnadaitalica} \label{verticaldefinition}
 For $[H^n,\dotsc,H^2,H^1]$ as above and $\xymatrix{g \ar@2{~>}@<0.25ex>[r]^{K^1} & g_1 \ar@2{~>}@<0.25ex>[r]^>>>>>>{K^2} & g_2 \dotsb  g_{m-1} \ar@2{~>}@<0.25ex>[r]^>>>>>>{K^m} & h}$, we define $$[K^m,\dotsc,K^2,K^1] \circ [H^n,\dotsc,H^2,H^1] = [K^m,\dotsc,K^2, K^1,H^n,\dotsc,H^2,H^1].$$
This is clearly well defined and associative. 
{Note that $[H^n,\dotsc,H^1] = [H^n] \circ \dotsb \circ [H^1]$.}

For 2-cells in $\cc{C}$, by Proposition \ref{prop:muinduceH} we have 
$[H_0^{\mu \circ \mu'}] = [H_0^{\mu'},\, H_0^{\mu}]$ $(\,= [H_0^{\mu'}] \circ [H_0^{\mu}]\,)$, and similarly for $H_1$.
\end{sinnadaitalica}

From Proposition \ref{prop:cuadradowhiskering} it follows:

\begin{proposition} \label{verticalok}
Let $H,\,\mu,\,\nu$ be as in \eqref{eq:H}, 
\ref{2-cellohomotopy} and \ref{homotopyo2-cell} respectively, and consider Definition \ref{Imu}. Then the following hold:

\begin{enumerate}
	\item $[\mu \circ H] = [I^\mu] \circ [H]$.
	
	\item $[H \circ \nu] = [H] \circ [I^\nu]$.\qed
\end{enumerate}

\end{proposition}

{{\bf Horizontal composition.}} We define now the horizontal composition in $\Ho$. We proceed as explained in \ref{sin:whiskerings}, that is, we will define it only between 2-cells and arrows, and show the axioms {\bf W}. 

\begin{sinnadaitalica} \label{horizontaldefinition}
For $\xymatrix{ X \ar@<1ex>[r]^{f} 
             \ar@<-1ex>[r]_{g} & Y \ar[r]^r & Y'}\in\C$ and $[H^n,\dotsc,H^2,H^1]: f \Rightarrow g$ as in Definition~\ref{def:2cellsdeHoC}, we define 
             $r \ast [H^n,\dotsc,H^2,H^1] = [r \ast H^n,\dotsc,r \ast H^2,r \ast H^1]$, and similarly for $X' \mr{\ell} X\in\C$ (see \ref{arrow*homotopy} and \ref{homotopy*arrow}). The fact that these formulas are well defined  follows from Proposition~\ref{prop:cuadradowhiskering}.
\end{sinnadaitalica}

Axiom {\bf W3} follows by definition.
To verify axiom {\bf W1}, it suffices to check the case in which the 2-cells are sequences of length 1, that is, given 
$\xymatrix{X \ar@<1.6ex>[r]^{f_1} 
             \ar@{}@<-1.3ex>[r]^{\!\! {[H]} \, \!\Downarrow}
             \ar@<-1.1ex>[r]_{f_2} & Y      
\ar@<1.6ex>[r]^{g_1} 
             \ar@{}@<-1.3ex>[r]^{\!\! {[K]} \, \!\Downarrow}
             \ar@<-1.1ex>[r]_{g_2} & Z}\in\Ho$, we have to check that $[g_2 \ast H,K \ast f_1] = [K \ast f_2,g_1 \ast H]$. Again this follows easily from Proposition~\ref{prop:cuadradowhiskering}, using axiom {\bf W1} in $\cc{D}$ for every 2-functor $(\C,\cc{W}) \mr{F} (\cc{D},\Theta_q)$.
             
For each $f$, we define the identity 2-cell of $\Ho$,
$id_f = [I^f]$, see Definition \ref{Imu} and recall the abuse $f = id_f$. By definition, it is immediate that $id_f$ is the identity for the vertical composition, and that axiom {\bf W2} is satisfied.

We define the identity arrows as in $\C$. 
It remains to define the associators and the unitors and check that they satisfy the axioms.  Before doing this it is convenient to construct the 2-functor $\cc{C} \mr{i} \Ho$.
 
\begin{sinnadaitalica} \label{functori}
On objects and arrows $\,i\,$ is just the identity. For a $2$-cell $\mu$ of $\C$, we define \mbox{$i\mu = [I^{\mu}]$,} that is the class of the sequence of length one given by any  $I^{\mu}$. From Definitions \ref{def:2cellsdeHoC} and \ref{Imu} it follows that for any homotopy $H$: 

 $\hspace{10ex}$  $i\mu = [H]$  $\iff$  for every 2-functor $(\cc{C},\, \cc{W}) \mr{F} (\cc{D},\,\Theta_q)$, $\widehat{FH} = F\mu$. 
\end{sinnadaitalica}

The unitors and the associator of $\Ho$ are obtained by applying $i$ to the ones of $\C$. Axioms {\bf N}$\lambda$, {\bf N}$\rho$ and {\bf N}$\theta$1-3 follow immediately from Proposition \ref{prop:cuadradowhiskering}, using the corresponding axioms in $\cc{D}$ for every 2-functor $(\C,\cc{W}) \mr{F} (\cc{D},\Theta_q)$.

We will now show that $\C \mr{i} \Ho$, mapping 
$\xymatrix{X \ar@<1.6ex>[r]^{f} 
             \ar@{}@<-1.3ex>[r]^{\!\! {\mu} \, \!\Downarrow}
             \ar@<-1.1ex>[r]_{g} & Y}$ to 
$\xymatrix{X \ar@<1.6ex>[r]^{f} 
             \ar@{}@<-1.3ex>[r]^{\!\! {i\mu} \, \!\Downarrow}
             \ar@<-1.1ex>[r]_{g} & Y}$, is a \mbox{2-functor}. From this fact, since the associators and the unitors of $\Ho$ are defined applying $i$ to the ones of $\C$, it will follow that they are invertible and that the pentagon and triangle identities hold, ending the proof that $\Ho$ is a bicategory.
             
With $i$ being trivial at the level of objects and arrows,  and mapping the identity 2-cells to the identities by definition, it suffices to check that $i$ preserves both compositions of 2-cells. The fact that $i$ preserves the vertical composition follows immediately by \ref{verticaldefinition}. To show that $i$ preserves the horizontal composition, we consider $\xymatrix{X \ar@<1.6ex>[r]^{f} 
             \ar@{}@<-1.3ex>[r]^{\!\! {\mu} \, \!\Downarrow}
             \ar@<-1.1ex>[r]_{g} & Y \ar[r]^r & Y'}\in\C$, and we have to show that $i(r \ast \mu) = r \ast i\mu$, i.e. that
            $[I^{r \ast \mu}] = [r \ast I^\mu]$.             
            For each 2-functor $(\C,\cc{W}) \mr{F} (\cc{D},\Theta_q)$, by Definition~\ref{Imu} and Proposition~\ref{prop:cuadradowhiskering} we have: 
$$\widehat{FI^{r \ast \mu}} = 
F(r \ast \mu) = Fr \ast F\mu =  Fr \ast \widehat{FI^\mu} = \widehat{F(r \ast I^\mu)},$$
showing the desired equation. The other case is similar. 
We have shown:

\begin{proposition} \label{funtoriandHo}
For any pair $(\C, \, \cc{W})$, $\Ho$ defined in \ref{bicategoryHo} 
is a bicategory, and \mbox{$\C \mr{i} \Ho$} defined in \ref{functori} is a 2-functor.  \qed
\end{proposition}

\begin{remark}
If we start with a 2-category $\cc{C}$, then $\Ho$ is also a 2-category.
\cqd
\end{remark}

\medskip

Using in order Remark \ref{rem:Hcomo3sinclases}, Proposition \ref{verticalok} and the definitions in \ref{horizontaldefinition}, \ref{functori} it follows:

\begin{proposition} \label{prop:Hcompde3}
Let $H$ be any homotopy as in Definition \ref{defhpy}. Then $[H]$ can be decomposed as:

$$
[H] = [\eps \circ (h * H^C) \circ \eta] = 
[I^\eps] \circ [h * H^C] \circ [I^\eta] =  
i\eps \circ h \ast [H^C] \circ i\eta.
$$

\vspace{-4ex}

\qed
\end{proposition}

We show now that the cylinder $C^{-1}$ and the homotopy $H^{-1}$
(see Definitions \ref{def:Cilindroinversaeidentidad}, \ref{def:Halamenos1} and \ref{def:construccionK}) yield actual inverses in $\Ho$.

\begin{proposition} \label{prop:homotinversible}
For any cylinder $C$,
$[H^C]$ is invertible in $\Ho$  and furthermore, 
$[H^C]^{-1} = [H^{(C^{-1})}]$.
\end{proposition}

\begin{proof}
For any 2-functor $(\cc{C},\, \cc{W}) \mr{F} (\cc{D},\,\Theta_q)$, we have $\widehat{FH^C} = \widehat{H^{FC}} = \widehat{FC}$, recall from Definition~\ref{beta} that $Fs * \widehat{FC} = \alpha$. Since we also have \mbox{$\widehat{FH^{(C^{-1})}} = \widehat{FC^{-1}}$, $Fs * \widehat{FC} = \alpha^{-1}$,} it follows using that $Fs$ is a quasiequivalence that $[H^C] \circ [H^{(C^{-1})}] = id_{d_1}$ and \mbox{$[H^{(C^{-1})}] \circ [H^C] = id_{d_0}$.}
\end{proof}

\begin{corollary} \label{invertibleH}
The class $[H]$ of any homotopy with  invertible cells is invertible in $\Ho$, and furthermore, $[H]^{-1} = [H^{-1}]$.
\end{corollary}
\begin{proof}
By Proposition \ref{prop:Hcompde3}, 
$[H] \circ [H^{-1}] = 
i\eps \circ h \ast [H^C] \circ i\eta
\circ 
i(\eta^{-1}) \circ h \ast [H^{(C^{-1})}] \circ i(\eps^{-1})$, which by Proposition \ref{prop:homotinversible} collapses to the identity. The other composition is similar.
\end{proof}

\noindent {\bf The universal property of $i$.} 
We will prove that, under some natural conditions on the class $\cc{W}$, the 2-functor \mbox{$\C \mr{i}\Ho$} is the localization of $\C$ with respect to $\cc{W}$. 
It should be noted that $i$ 
has the universal property of making the arrows of $\cc{W}$ into equivalences in a \emph{strict} 2-categorical sense. {By this we mean that, for any bicategory $\cc{D}$, precomposition with $i$ yields an isomorphism of the appropriate Hom-bicategories, and not just a biequivalence (see details in Definition \ref{localizacion} below).} 
We state precisely what we mean by localization of $\C$ with respect to $\cc{W}$ (See Section \ref{sec:intro} for a comparison with \cite{PRONK2}): 


\begin{definition}\label{localizacion} 
	A pseudofunctor $(\C,\cc{W})\mr{i} (\cc{E},\Theta)$ is the localization of $\C$ with respect to $\cc{W}$ if it is universal in the following sense: For any bicategory $\cc{D}$, precomposition with $i$, \mbox{$\Hom(\cc{E},\cc{D}) \mr{i^*} \Hom_{(\cc{W},\Theta)}(\C,\cc{D})$} is a biequivalence of bicategories, where $\Hom_{(\cc{W},\Theta)}(\C,\cc{D})$ stands for the full subbicategory of $\Hom(\C,\cc{D})$ consisting of those pseudofunctors that map the arrows of $\cc{W}$ to equivalences. When $i^*$ is an isomorphism, we say that $i$ is a {\em strict} localization or a localization in a strict sense.
\end{definition}

We begin by stating and proving various results which lead to Theorem \ref{teo:mediapuposta} and Corollary~\ref{teo:previo pu}. This theorem is proven \emph{without any hypothesis on $\cc{W}$}, and shows that $i^*$ will be an isomorphism of bicategories as soon as it takes its values in the subbicategory $\Hom_{(\cc{W},\Theta)}(\C,\cc{D})$. Then we show that, under two natural conditions on $\cc{W}$, $i$ maps the arrows of $\cc{W}$ to equivalences, and thus the desired result follows.

For a cylinder $C$ and a 2-functor $F: (\cc{C},\cc{W}) \mr{} (\cc{D},\Theta_q)$, recall that $\widehat{FH^{C}} = \widehat{FC}$, which is the unique 2-cell such that $Fs * \widehat{FC} = F\alpha$.

\begin{lemma} \label{lema:paramediapu}
For any cylinder $C$, we have $[s * H^C] = i\alpha$
\end{lemma}

\begin{proof}
Let $F: (\cc{C},\cc{W}) \mr{} (\cc{D},\Theta_q)$ be a 2-functor, we compute using Proposition~\ref{prop:cuadradowhiskering} $\widehat{F(s * H^C)} = Fs * \widehat{FH^C} = Fs * \widehat{FC} = F\alpha$ and the Lemma follows by \ref{functori}.
\end{proof}

\begin{proposition} \label{prop:mediapu}
Let $\Ho \mr{G} \cc{D}$ be a 2-functor such that 
$Gi\cc{W} \subseteq \Theta_q$, and let $H$ be a homotopy. Then \mbox{$G[H] = \widehat{GiH}$.} 
\end{proposition}
\begin{proof} 
Consider $F = Gi: (\cc{C},\cc{W}) \mr{} (\cc{D},\Theta_q)$. Note that $F$ equals $G$ on objects and arrows. 
Let $C$ be the cylinder of $H$. From Lemma~\ref{lema:paramediapu} and the definition in \ref{horizontaldefinition}
it follows $s \ast [H^{C}] = i\alpha$. Applying $G$ we have $Fs * G[H^{C}] = F\alpha$, and thus $G[H^{C}] = \widehat{F(H^{C})}$. 
Then we compute, using Proposition \ref{prop:Hcompde3}, functoriality of $G$, and Remark \ref{def:FyHinducen2cell}:

\smallskip

$G[H] 
 = Gi\varepsilon \circ Gh \ast G[H^{C}] \circ Gi\eta
 = F\varepsilon \circ Fh \ast \widehat{F(H^{C})} \circ F\eta = \widehat{FH}$.
\end{proof}
\begin{corollary} \label{idetermina}
Let $\Ho \mr{G} \cc{D}$ be any 2-functor such that 
$Gi\cc{W} \subseteq \Theta_q$. Then, the composite $Gi$  completely determines $G$.
\end{corollary}
\begin{proof}
Since $\, i \,$ is trivial at the level of objects and arrows clearly $GX = GiX$ and \mbox{$Gf = Gif$.} 
The computation $G[H^n,...,H^1] = 
G[H^n] \circ ... \circ G[H^1] = 
\widehat{GiH^n} \circ ... \circ \widehat{GiH^1}$ (which follows from Proposition~\ref{prop:mediapu})
finishes the proof.
\end{proof}

\begin{theorem} \label{2-functorcase} 
Let $\C \mr{i} \Ho$ be the 2-functor in \ref{funtoriandHo}. Then,   
for any bicategory $\cc{D}$ and any 2-functor $(\C,\cc{W}) \mr{F} (\cc{D},\Theta_q)$, there is a unique extension of $F$ to $\Ho$. That is, there is a 2-functor $G: \Ho \mr{} \cc{D}$, unique such that $G i = F$. Note that by Proposition 
\ref{prop:mediapu} the value of $G$ on the class of a homotopy $H$ is necessarily 
$\widehat{FH}$. 
\end{theorem}

\begin{proof}
By Corollary \ref{idetermina} we have that the unique possible definition of $G$ is $GX = FX$, 
$Gf  = Ff$ and $G[H^n,...,H^1]  = 
\widehat{FH^n} \circ ... \circ \widehat{FH^1}$. By the definition $i\mu = [I^{\mu}]$ and Definition~\ref{Imu} it follows that $Gi\mu = F\mu$ for any 2-cell $\mu$ of $\cc{C}$. It only remains to show that $G$ is a 2-functor.

Clearly the functoriality of $G$ on objects and arrows holds since $Gi = F$  and $\, i \,$ is trivial. The functoriality for the vertical composition of 2-cells holds by \ref{verticaldefinition}. For the horizontal composition we proceed as explained in \ref{sin:whiskerings}, that is, we consider only horizontal compositions between 
2-cells and arrows. 
It suffices to check this on 2-cells given by a single homotopy. 
Let $r$ and $H$ as in \ref{arrow*homotopy}, recall \ref{horizontaldefinition} and Proposition~\ref{prop:cuadradowhiskering}. Then:
$$
G(r \ast [H]) = G[r \ast H] = \widehat{F(r \ast H)} = 
F(r)  \ast \widehat{FH} = G(r) \ast G[H].
$$
The case $[H] \ast \ell$ is similar. 
\end{proof}

\begin{remark} \label{mejortenerlo}	
In the situation of the theorem above, for $\xi,\xi_1,\dotsc,\xi_n$ 2-cells of $\Ho$, we have: 
$$ \xi = [H^n] \circ \dotsb \,[H^2] \circ [H^1] 
\, \iff \, G(\xi) = 
\widehat{FH^n} \circ \dotsb \widehat{FH^2} \circ \widehat{FH^1}$$
$$ [H] = \xi_n \circ \dotsb \xi_2 \circ \xi_1 
\, \iff \, \widehat{FH} = 
G(\xi_n) \circ \dotsb G(\xi_2)  \circ G(\xi_1).$$ 
\vspace{-6ex}

\qed
\end{remark}

We pass now to prove the general case of Theorem 
\ref{2-functorcase} for pseudofunctors. 
Let  $(\C,\cc{W}) \mr{F} (\cc{D},\Theta_q)$ be a pseudofunctor, and consider its factorization through $\C_F$ as in \ref{sin:primerfactorizacion}. {Let $\Ho \mr{G} \cc{D}$ be any pseudofunctor such that $G i = F$},   
applying Proposition \ref{prop:upoffactoriz} we have that there is a 
unique 2-functor $E$ such that $F_2 = Ei$ and 
$G = F_1 E$, as in the following diagram:
$$
\xymatrix{\C \ar[rr]^F \ar[rd]^{i} \ar@/_2ex/[ddr]_{F_2} && \cc{D} \ar@{<-}[dl]_{G} \ar@{<-}@/^2ex/[ddl]^{F_1} \\
& \Ho \ar@{-->}[d]^{E} \\ & \C_F  }
$$

{The reader should note that this construction is independent of Theorem \ref{2-functorcase}, which given $F_2$ also yields a unique 2-functor $E$ such that $F_2 = Ei$. The fact that this is thus the same 2-functor $E$ is the ``trick" that allows to prove item 1 of Theorem \ref{teo:mediapuposta} below.}

\vspace{1ex}

{
\begin{theorem} \label{teo:mediapuposta}
Let $\C \mr{i} \Ho$ be the 2-functor in \ref{funtoriandHo}. Then precomposing with $i$ establishes a biequivalence of bicategories, which in fact is an isomorphism:
$$
\Hom_{(i\cc{W},\, \Theta)}(\cc{H}o(\cc{C}, \cc{W}),\, \cc{D})
\mr{i^*}
\Hom_{(\cc{W},\, \Theta)}(\cc{C},\, \cc{D})
$$
\end{theorem}}
\begin{proof} 
{The statement can be divided in the following three items:} 

1.  For any bicategory $\cc{D}$ and any pseudofunctor $(\C,\cc{W}) \mr{F} (\cc{D},\Theta)$, there exists a unique extension of $F$ to $\Ho$. That is, there is a pseudofunctor $\Ho \mr{F'} \cc{D}$, unique such that $F' i = F$
(which clearly then maps $i\cc{W}$ to $\Theta$). Furthermore,
the value of $F'$ on the class of a homotopy $H$ is $\widehat{FH}$, that is, $F'[H] = \widehat{FH}$. 
 
2. For every pseudonatural transformation $F \Mr{\theta} G: (\C,\cc{W}) \mr{} (\cc{D},\Theta)$ there is a pseudonatural transformation $F' \Mr{\theta'} G'$ unique such that 
$\theta' i = \theta$. 

3. For every modification $\theta \mr{\rho} \eta: F \Mr{} G: (\C,\cc{W}) \mr{} (\cc{D},\Theta)$ there is a modification $\theta' \mr{\rho} \eta'$ unique such that $\rho' i = \rho$. 

\vspace{1ex}

\emph{1.} Let $F = F_1F_2$ be the factorization through 
$(\C_F,\Theta_q)$ as explained in \ref{sin:primerfactorizacion}, \ref{quasiequivalencias}. Let $F'_2$ be the extension of $F_2$ given by Theorem \ref{2-functorcase}. Set $F' = F_1F'_2$. Then, 
\mbox{$F'i =  F_1F'_2 i = F_1F_2 = F$.} The uniqueness of $F'$ is given by Proposition \ref{prop:upoffactoriz} plus the uniqueness of $F'_2$. For the second statement we compute  
\mbox{$F'[H] = F_1F'_2[H] = 
F_1 (\widehat{F_2H}) = \widehat{FH}$,} this last equality given by Proposition  \ref{prop:bastaconlos2functors}. 

\vspace{1ex}

\emph{2.} 
Since $i$ is the identity at the level of objects and arrows, the only possible definitions are $\theta'_X = \theta_X$, $\theta'_f = \theta_f$ for every $X,f$. 
Since the structural morphisms of $F'$ (resp. $G'$) are those of $F$ (resp. $G$), axioms {\bf PN0} and {\bf PN1} 
for $\theta'$ are equivalent to those for $\theta$. For axiom {\bf PN2}, we have to show the following equation for every homotopy $\xymatrix{ f \ar@2{~>}@<0.25ex>[r]^{H} & g}$ as in Definition~\ref{defhpy}:

$$\vcenter{\xymatrix@C=-0pc{
                      {G}{f} \dcellbymedio{\widehat{{G}H}}  
		      &&
		      \theta_{X}  \did 
		      \\
		       {G}{g} \dl
		       &&
		      \theta_{X} \dr \ar@{}[dll]|{\theta_{g}}
		      \\
		       \theta_{{Y}} 
		       && 
		      {F}{g} 
		      }}
     \vcenter{\xymatrix@C=-.4pc{\quad = \quad }}
      \vcenter{\xymatrix@C=-0pc{
		      {G}{f} \dl
		      &&
		      \theta_{X} \dr \ar@{}[dll]|{\theta_{f}} 
		      \\
		      \theta_{{Y}} \did
		      &&
		      {F}{f} \dcellbymedio{\widehat{{F}H}} 
		      \\ 
		      \theta_{Y}
		      &&
		      {F}{g} 
		      }}$$

\noindent That is, by the definition in formula \eqref{eq:FHsombrero}, 

\begin{equation} \label{eq:qvq3}
\vcenter{\xymatrix@C=-1pc@R=1.5pc{
& Gf \dcellb{G\eta} && \quad\quad & \theta_{X}  \did \\
& G(h\ast d_0) \op{\phi} &&& \theta_{X}  \did \\
Gh \did && Gd_0 \dcellbbis{\widehat{GC}} && \theta_{X}  \did \\
Gh && Gd_1 && \theta_{X}  \did \\
& G(h\ast d_1) \cl{\phi} \dcellb{G\eps} &&& \theta_{X}  \did  \\
& {G}{g} \dl & \ar@{}[dr]|{\theta_{g}} && \theta_{X} \dr   \\
& \theta_{{Y}} &&& {F}{g} }}
     \vcenter{\xymatrix@C=-.4pc{\quad = \quad }}
\vcenter{\xymatrix@C=-1pc@R=1.5pc{
		      {G}{f} \dl
		      & \quad\quad & \ar@{}[d]|{\theta_{f}} &
		      \theta_{X} \dr  
		      \\
\theta_{{Y}} \did && & Ff \dcellb{F\eta} \\
\theta_{{Y}} \did && & F(h\ast d_0) \op{\phi} \\
\theta_{{Y}} \did && Fh \did && Fd_0 \dcellbbis{\widehat{FC}} \\
\theta_{{Y}} \did && Fh && Fd_1 \\
\theta_{{Y}} \did && & F(h\ast d_1) \cl{\phi} \dcellb{F\eps} \\
\theta_{{Y}} && & Fg }}
\end{equation}

\noindent where $Fs \ast_F \widehat{FC} = F\alpha$ and $Gs \ast_G \widehat{GC} = G\alpha$. 
Using axiom {\bf PN2} for $\theta$ on the 2-cell $\alpha$ we have:

\begin{equation} \label{eq:arevertir}
\vcenter{\xymatrix@C=-1pc@R=1.5pc{
& G(s\ast d_0) \op{\phi} && \quad\; & \theta_{X}  \did \\
Gs \did && Gd_0 \dcellbbis{\widehat{GC}} && \theta_{X}  \did \\
Gs && Gd_1 && \theta_{X}  \did \\
& G(s\ast d_1) \cl{\phi} \dl & \dcr{\theta_{s\ast d_1}} && \theta_{X} \dr  \\
& \theta_{{Z}} &&& F(s\ast d_1)  }}  
\vcenter{\xymatrix@C=0ex{ \; = \; }}
\vcenter{\xymatrix@C=-1pc@R=1.5pc{
 G(s\ast d_0) \dl & \dcr{\theta_{s\ast d_0}} && \dr \theta_{X}  \\
 \theta_Z \did & \quad\; && F(s\ast d_0) \op{\phi} \\
\theta_Z \did && Fs \did && Fd_0 \dcellbbis{\widehat{FC}} \\
\theta_Z \did && Fs && Fd_1 \\
 \theta_{{Z}} &&& F(s\ast d_1) \cl{\phi} }}
\end{equation}

Using axiom {\bf PN1} twice (for the arrows $s, d_0$ and the arrows $s, d_1$)  the first equality below follows from \eqref{eq:arevertir}:

$$\vcenter{\xymatrix@C=0pc@R=1.5pc{
Gs \did && Gd_0 \dcellbbis{\widehat{GC}} && \theta_{X}  \did \\
Gs \did && Gd_1 \dl & \dc{\theta_{d_1}} & \dr \theta_{X} \\
Gs \dl & \dc{\theta_s} & \dr \theta_C && Fd_1 \did \\
\theta_Z && Fs && Fd_1}}
\vcenter{\xymatrix@C=0ex{ \; = \; }}
\vcenter{\xymatrix@C=0pc@R=1.5pc{
Gs \did && Gd_0 \dl & \dc{\theta_{d_0}} & \dr \theta_{X} \\
Gs \dl & \dc{\theta_s} & \dr \theta_{C} && Fd_0 \did \\ 
\theta_Z \did && Fs \did && Fd_0 \dcellbbis{\widehat{FC}} \\
\theta_Z && Fs && Fd_1 }}
\vcenter{\xymatrix@R=0pc{\eqref{ascensor}  \\ = }}
\vcenter{\xymatrix@C=0pc@R=1.5pc{
Gs \did && Gd_0 \dl & \dc{\theta_{d_0}} & \dr \theta_{X} \\
Gs \did && \theta_{C} \did && Fd_0 \dcellbbis{\widehat{FC}} \\ 
Gs \dl & \dc{\theta_s} & \dr \theta_{C} && Fd_1 \did \\
\theta_Z && Fs && Fd_1 }}
$$

Since $\theta_s$ is invertible, and $Gs$ is a quasiequivalence, it follows

\begin{equation} \label{eq:qvq}
\vcenter{\xymatrix@C=0ex{Gd_0 \dcellbbis{\widehat{GC}} && \theta_X \did \\
Gd_1 \dl & \dc{\theta_{d_1}} & \dr \theta_X \\
\theta_C && Fd_1   }}
\vcenter{\xymatrix@C=0ex{ \; = \; }}
\vcenter{\xymatrix@C=0ex{Gd_0 \dl & \dc{\theta_{d_0}}  & \dr \theta_X \\
\theta_C \did && Fd_0 \dcellbbis{\widehat{FC}} \\
\theta_C && Fd_1 }}
\end{equation}

Now, we reverse the path that took us from \eqref{eq:arevertir} to \eqref{eq:qvq}, but with $h$ instead of $s$. First we compose \eqref{eq:qvq} with $Gh$ and $\theta_h$ and use \eqref{ascensor}, it follows

$$\vcenter{\xymatrix@C=0pc@R=1.5pc{
Gh \did && Gd_0 \dcellbbis{\widehat{GC}} && \theta_{X}  \did \\
Gh \did && Gd_1 \dl & \dc{\theta_{d_1}} & \dr \theta_{X} \\
Gh \dl & \dc{\theta_h} & \dr \theta_C && Fd_1 \did \\
\theta_Y && Fh && Fd_1}}
\vcenter{\xymatrix@C=0ex{ \; = \; }}
\vcenter{\xymatrix@C=0pc@R=1.5pc{
Gh \did && Gd_0 \dl & \dc{\theta_{d_0}} & \dr \theta_{X} \\
Gh \dl & \dc{\theta_h} & \dr \theta_{C} && Fd_0 \did \\ 
\theta_Y \did && Fh \did && Fd_0 \dcellbbis{\widehat{FC}} \\
\theta_Y && Fh && Fd_1 }}$$

Using axiom {\bf PN1} as above, it follows 

\begin{equation} \label{eq:ausar}
\vcenter{\xymatrix@C=-1pc@R=1.5pc{
& G(h\ast d_0) \op{\phi} && \quad\; & \theta_{X}  \did \\
Gh \did && Gd_0 \dcellbbis{\widehat{GC}} && \theta_{X}  \did \\
Gh && Gd_1 && \theta_{X}  \did \\
& G(h\ast d_1) \cl{\phi} \dl & \dcr{\theta_{h\ast d_1}} && \theta_{X} \dr  \\
& \theta_{{Y}} &&& F(h\ast d_1)  }}  
\vcenter{\xymatrix@C=0ex{ \; = \; }}
\vcenter{\xymatrix@C=-1pc@R=1.5pc{
 G(h\ast d_0) \dl & \dcr{\theta_{h\ast d_0}} && \dr \theta_{X}  \\
 \theta_Y \did & \quad\; && F(h\ast d_0) \op{\phi} \\
\theta_Y \did && Fh \did && Fd_0 \dcellbbis{\widehat{FC}} \\
\theta_Y \did && Fh && Fd_1 \\
 \theta_{{Y}} &&& F(h\ast d_1) \cl{\phi} }}
 \end{equation}

Finally we compute, starting from the left side in \eqref{eq:qvq3}:

$$\vcenter{\xymatrix@C=-1pc@R=1.5pc{
& Gf \dcellb{G\eta} && \quad\quad & \theta_{X}  \did \\
& G(h\ast d_0) \op{\phi} &&& \theta_{X}  \did \\
Gh \did && Gd_0 \dcellbbis{\widehat{GC}} && \theta_{X}  \did \\
Gh && Gd_1 && \theta_{X}  \did \\
& G(h\ast d_1) \cl{\phi} \dcellb{G\eps} &&& \theta_{X}  \did  \\
& {G}{g} \dl & \ar@{}[dr]|{\theta_{g}} && \theta_{X} \dr   \\
& \theta_{{Y}} &&& {F}{g} }}
     \vcenter{\xymatrix@R=0pc{{\bf PN2} \\ = }}
\vcenter{\xymatrix@C=-1pc@R=1.5pc{
& Gf \dcellb{G\eta} && \quad\quad & \theta_{X}  \did \\
& G(h\ast d_0) \op{\phi} &&& \theta_{X}  \did \\
Gh \did && Gd_0 \dcellbbis{\widehat{GC}} && \theta_{X}  \did \\
Gh && Gd_1 && \theta_{X}  \did \\
& G(h\ast d_1) \cl{\phi} \dl & \dcr{\theta_{h\ast d_1}} && \theta_{X} \dr  \\
& \theta_{{Y}} \did &&& F(h\ast d_1) \dcellb{F\eps}   \\
& \theta_{{Y}} &&& {F}{g} }}  
\vcenter{\xymatrix@R=0pc{\eqref{eq:ausar}  \\ = }}
\vcenter{\xymatrix@C=-1pc@R=1.5pc{
 Gf \dcellb{G\eta} & \quad && \theta_{X}  \did \\
 G(h\ast d_0) \dl & \dcr{\theta_{h\ast d_0}} && \dr \theta_{X}  \\
 \theta_Y \did &&& F(h\ast d_0) \op{\phi} \\
\theta_Y \did && Fh \did && Fd_0 \dcellbbis{\widehat{FC}} \\
\theta_Y \did && Fh && Fd_1 \\
 \theta_{{Y}} \did &&& F(h\ast d_1) \cl{\phi} \dcellb{F\eps}   \\
 \theta_{{Y}} &&& {F}{g} }},
 $$

\noindent which equals the right side in \eqref{eq:qvq3} by {\bf PN2}. 

\vspace{1ex}

\emph{3.} 
Since $i$ is the identity at the level of objects, the only possible definition is $\rho'_X = \rho_X$. Since for any arrow $f$ of $\cc{C}$, by the proof of item 2 we have $\theta'_f = \theta_f$ and $\mu'_f = \mu_f$, the equality in axiom {\bf PM} is the same either for $\rho$ or for $\rho'$.
\end{proof}

\begin{remark}\label{dosclases}
Even though this is not needed for the localization result in this article, it is a natural idea to restrict the cylinders to have the map $W \mr{} Z$ in an arbitrary subclass $\cc{W}_0 \subset \cc{W}$ containing the identities, much like {\em fibrant} cylinders are considered classically.
It is easy to check that everything goes exactly in the same way with this added generality, and that we can construct a bicategory
$\cc{H}o^{\cc{W}_0}(\cc{C}, \cc{W})$ that has the same property, {that is precomposition with $i$ establishes an isomorphism of bicategories:}
$$
Hom_{(i\cc{W},\, \Theta)}(\cc{H}o^{\cc{W}_0}(\cc{C}, \cc{W}),\, \cc{D})
\mr{i^*}
Hom_{(\cc{W},\, \Theta)}(\cc{C},\, \cc{D})
$$ 

\vspace{-5ex}

\cqd
\end{remark} 

\begin{corollary}[{of Theorem \ref{teo:mediapuposta}}] \label{teo:previo pu}
Assume $\C \mr{i} \Ho$ maps the arrows of $\cc{W}$ to equivalences. Then, it is the strict localization with respect to $\cc{W}$ in the sense stated in Definition~\ref{localizacion}, {that is precomposition with $i$ establishes an isomorphism of bicategories:}
$$
Hom(\cc{H}o(\cc{C}, \cc{W}),\, \cc{D})
\mr{i^*}
Hom_{(\cc{W},\, \Theta)}(\cc{C},\, \cc{D}).
$$
\end{corollary}       
\begin{proof}
By the assumption {it follows that} $i^*$ takes values in the subcategory $\Hom_{(\cc{W},\Theta)}(\C,\cc{D})$. 
Theorem \ref{teo:mediapuposta} states that it is an isomorphism of bicategories.
\end{proof}
 
\vspace{1ex}

We proceed now to consider two natural conditions on the class $\cc{W}$ which are sufficient to ensure that the assumption in Corollary \ref{teo:previo pu} holds.

\begin{definition} \label{def:3x2}
We say that the class $\cc{W}$ satisfies  the \emph{3 for 2}\footnote{A. Joyal suggested this terminology (instead of the usual 2 for 3) because \textquoteleft you pay for 2 and get 3\textquoteright.} property if for every three arrows $f,g,h$ such that there is an invertible 2-cell $gf \cong h$, whenever two of the three arrows are in $\cc{W}$, so is the third one.
\end{definition}

\begin{definition} \label{def:retrsect}
Let $X \mr{s} Y$, $Y \mr{r} X\in\C$. If there is an invertible 2-cell $r\ast s \cong id_X$, $s$ is called a section for $r$, and $r$ is called a retraction for $s$. An arrow $X \mr{s} Y$ is called a section if there exists $r$ such that $s$ is a section for $r$ and dually an arrow is called a retraction if it admits a section. An arrow that is either a section or a retraction is called a split arrow.
\end{definition}

\begin{proposition} \label{prop:otramediapu}
Assume $\cc{W}$ satisfies the 3 for 2 property. Then any split arrow in $\cc{W}$
is mapped to an equivalence by $\C \mr{i} \Ho$. 
\end{proposition}

\begin{proof} 
Let $X \mr{s} Y$, $Y \mr{r} X\in\C$ and an invertible 2-cell $r\ast s \Mr{\alpha} id_X\in\C$.
Note that by the 3 for 2 property $r$ is in $\cc{W}$ if and only if $s$ is. 
Since we already have \mbox{$r\ast s \Mr{i\alpha} id_X$} in $\Ho$, it remains to show that we have an invertible $2$-cell
 $s\ast r \Mr{} id_Y$. 
Consider the diagram 
$\vcenter{\xymatrix@C=1.5pc@R=1.5pc{Y \ar[rr]^{s * r} \ar[dd]_{id_Y} \ar[rrdd]|{\comw{M^M} r \comw{M^M}} & \ar@{}[dr]|{\Downarrow \; \alpha * r \;} & Y \ar[dd]|{\circ}^{r}\\
\ar@{}[dr]|{ \; \; \Uparrow \; id_r} && \\ 
Y \ar[rr]|{\circ}_{r}  && X}}$ 
(as in Definition \ref{defcylestiloQuillen}) 
which defines the cylinder \mbox{$C=(Y,X,s*r,id_Y,r,r,\alpha *r,id_r)$.} Thus by Proposition \ref{prop:homotinversible} we have the desired invertible 2-cell $[H^C]$.
\end{proof}
    

\begin{corollary} \label{cor:pu}
Assume that $\cc{W}$ satisfies the 3 for 2 property, and that any arrow of $\cc{W}$ can be written (up to isomorphism) as a composition of split arrows of $\cc{W}$. Then the 2-functor $\C \mr{i} \Ho$ maps the arrows of $\cc{W}$ to equivalences.  \qed
\end{corollary}       

\medskip 
Clearly putting together Corollaries \ref{cor:pu} and \ref{teo:previo pu} we have the main result of this article, a construction of the strict bicategorical localization with respect to a family of arrows satisfying the aforementioned properties. 
As mentioned in the introduction, these properties are satisfied by the weak equivalences between fibrant-cofibrant objects of a model bicategory \cite{DDS.model_bicat}.

\begin{theorem} \label{enfin} 
If $\cc{W}$ satisfies the 3 for 2 property, and any arrow of $\cc{W}$ can be written (up to isomorphism) as a composition of split arrows of $\cc{W}$, then the 2-functor $\C \mr{i} \Ho$ is the strict localization with respect to $\cc{W}$ in the sense stated in Definition~\ref{localizacion}. \qed
\end{theorem}

\vspace{-.2cm}

\bibliographystyle{unsrt}

\begin{thebibliography}{99}
\vspace{-.2cm}
\bibitem{DDS.model_bicat} Descotte M.E., Dubuc, E., Szyld M., \textsl{Model bicategories and their homotopy bicategories}, arXiv:1805.07749 (2018).
\bibitem{TesisEmi} Descotte M.E., \textsl{A theory of 2-pro-objects, a theory of 2-model 2-categories and the 2-model structure for $2$-$\cc{P}ro(\C)$}, PhD Thesis, Universidad de Buenos Aires (2015).
\bibitem{2-pro} Descotte M.E., Dubuc, E., \textsl{A theory of 2-pro-objects}, Cahiers de Topologie et G\'eom\'etrie Diff\'erentielle Cat\'egoriques,  LV-1 (2014), and with expanded proofs arXiv:1406.5762.
\bibitem{DHKS} Dwyer W.G., Hirschhorn P.S., Kan D.M., Smith J.H., \textsl{Homotopy Limit Functors on Model Categories and Homotopical Categories}, AMS Mathematical Surveys and Monographs 113 (2004).
\bibitem{GZ} Gabriel P., Zisman M., \textsl{Calculus of fractions and homotopy theory}, Ergebnisse der Mathematik und ihrer Grenzgebiete. 2. Folge
Volume 35 (1967).
\bibitem{basicbicat} Leinster T. \textsl{Basic Bicategories}, arXiv:math/9810017 (1998).
\bibitem{McL} Mac Lane S. \textsl{Categories for the Working Mathematician}, Graduate Texts in Mathematics Volume 5 (1971).
\bibitem{PRONK2}  Pronk, D. A. \textsl{Etendues and stacks as bicategories of fractions}, Compositio Mathematica Volume 102 no. 3 Pages 243-303 (1996). 
\bibitem{PronkWarren}
Pronk, D. A., Warren M. A., \textsl{Bicategorical fibration structures and stacks, }
Theory and Applications of Categories 29 (2014).
\bibitem{Quillen} Quillen D., \textsl{Homotopical Algebra}, Springer Lecture Notes in Mathematics 43 (1967). 
\bibitem{S.paper11} Szyld M., \textsl{The homotopy relation in a category with weak equivalences}, arXiv:1804.04244 (2018).
\end{thebibliography}

\vspace{-.3cm}.

\begin{appendix}

\section{On the equivalence relation of homotopies.}
It is a natural question if it would be possible to give a \emph{syntactical} definition rather than the \emph{semantical} one for the equivalence relation given in Definition \ref{def:primerdefadhoc}. This is a delicate point, that can be traced back to Quillen's original
work on model categories (see (*) below). As test work for our construction on bicategories we have considered the $2$-localization of a model category using homotopies, and the possibility of giving such a definition in this case, that we call the \emph{germ relation}. Following a question by the referee, we resume here our findings, that probably can be generalized with some work to the general case of a bicategory. 

{Consider homotopies and cylinders as in Definition \ref{defhpy}, but with equalities in place of $2$-cells.}
It is clear and straightforward how to define a morphism of cylinders and to check that they compose. Such a morphism consists of a pair of arrows $W \mr{\phi} W'$,  $Z \mr{\varphi} Z'$ satisfying the evident equations with respect to the cylinder structures. 
 Given two homotopies $H$ and $H'$, we say that they are related if there exists a morphism of cylinders 
$C \xr{(\phi,\, \varphi)} C'$ such that $h' \circ \phi = h$. We call the equivalence relation generated by this relation the \emph{germ relation}, because we are actually computing a colimit of sets: recall diagram \eqref{eq:H}, $f = h*d_0,\; g = h*d_1$. Consider the data  
$X,\;Z,\;Y,\;x,\;f,\;g$ fixed, and all possible 
$W,\; d_0,\; d_1$:
$$ \label{1_eq:H}
\xymatrix@C=8ex
   {
    X \amalg X \ar@<-.5ex>[dr]_{\binom{x}{x}} 
      \ar[r]_{\binom{d_0}{d_1}}
      \ar@/^5ex/[rr]_{\binom{f}{g}}  
   & W \ar[r]^h \ar[d]|{\circ}^s 
   & Y 
   \\ 
   & Z} 
$$
We imagine $\binom{x}{x}$ to be a point of $Z$, $W$ to be a neighborhood of $\binom{x}{x}$, and $h$ a function sending 
$\binom{x}{x}$ to the point $\binom{f}{g}$ of $Y$. This defines a functor into the category of sets, and the equivalence classes are the elements of the colimit in $\cc{S}et$.

\vspace{1ex}

It is easy to see that if two homotopies are related by cylinder morphisms, then they are equivalent as in Definition \ref{def:primerdefadhoc} (so the germ relation identifies less). 
In our work, finding these morphisms is the way in which we have usually shown that different homotopies are equivalent.

Initially we hoped that the germ relation would allow to define a 2-category of classes of homotopies, but all our attempts were blocked near that purpose. 
To put it shortly, we could define the data of a 2-category, and prove all 
the axioms, except for the interchange law {\bf H2} relating both compositions. This shows that a coarse relation identifying both sides of the law is necessary. Our definition of the equivalence relation of homotopies in this paper can be seen as a solution to this problem.

(*) Quillen also defines an equivalence relation between homotopies (see  
\cite[Ch.I $\sharp \,2$]{Quillen}) and shows that the equivalence  classes can be composed vertical and horizontally, but he does not mention the compatibility between both compositions, a condition that he certainly was aware of. We have checked that the relation considered by Quillen is actually our germ relation disguised in a different form.
This suggests that the interchange law does not hold for the equivalence classes of the germ relation.

\section{On vertical composition of 
homotopies.} \label{sec:vertcomp}

It this appendix we consider the question of vertical composition of homotopies, in other words, if, for any two composable homotopies $f \mrhpy{H} g \mrhpy{K} h$, we can find a single 
homotopy representing the class $[K,H]$. 
 The following lemma gives certain conditions under which this is the case, which are satisfied for the weak equivalences between fibrant-cofibrant objects of a model bicategory \cite{DDS.model_bicat}. 
The reader will probably recognize here an abstract setting corresponding to Quillen's proof of the transitivity of the left homotopy relation in \cite[Lemma 3]{Quillen}.

\begin{lemma} \label{lema:fittingparacompvert}
Assume that we have 
$
\xymatrix@C=7ex{X
\ar[r]^{f_1,f_2,f_3} & Y}$, and homotopies 
\mbox{$f_1 \mrhpy{H^1} f_2 \mrhpy{H^2} f_3$} as in Definition~\ref{defhpy}, with 
$Z^1 = Z^2 = Z$, $x^1 = x^2 = x$
fitting in the following diagram, where 
{$\nu^1,\, \nu^2$, $\gamma^1$, $\gamma^2$} are invertible 2-cells.         
$$\xymatrix@C=8ex{& W^1 \ar[rd]_{b^1} \ar@/^2ex/[rrd]|{\circ}^{s^1}_>>>>>>>>>>{\cong \, \nu^1} \ar@/^4ex/[rrrd]^{h^1}_>>>>>>>>>>>>>{\cong \,  \gamma^1} \\
X \ar[ru]^{d_1^1} \ar[rd]_{d^2_0} \ar@{}[rr]|{\Downarrow \delta} && W \ar[r]|{\circ}^{s} \ar@/_3ex/[rr]^>>>>>>>{h} & Z & Y \\
& W^2 \ar[ru]^{b^2} \ar@/_2ex/[rru]|{\circ}_{s^2}^>>>>>>>>>>{\cong \,  \nu^2} \ar@/_4ex/[rrru]_{h^2}^>>>>>>>>>>>>>{\cong \,  \gamma^2}}$$
Assume also that  
\smallskip

\noindent {\bf 1.}
The 2-cell $h^1 \ast d_1^1 \Mr{\eps^1} f_2 \Mr{\eta^2} h^2 \ast d^2_0$ equals 
$h^1 \ast d_1^1 \Xr{\gamma^1 \ast d_1^1} h \ast b^1 \ast d_1^1 \Xr{h \ast \delta} h \ast b^2 \ast d^2_0 \Xr{\gamma^2 \ast d^2_0} h^2 \ast d^2_0,$

\smallskip

\noindent {\bf 2.}
The 2-cell 
$s^1 \ast d_1^1 \Xr{(\alpha^1_1)} x \Xr{(\alpha^2_0)^{-1} \!\!\!\!\!\!\!\!} s^2 \ast d^2_0$ equals 
$s^1 \ast d_1^1 \Xr{\nu^1 \ast d_1^1} s \ast b^1 \ast d_1^1 \Xr{s \ast \delta} s \ast b^2 \ast d^2_0 \Xr{\nu^2 \ast d^2_0} s^2 \ast d^2_0.$

\smallskip

Then there exists a homotopy $H$ from $f_1$ to $f_3$ such that $[H] = [H^2,H^1]$.

\smallskip

Furthermore, $H$ can be constructed as follows:
consider first the cylinder $C$ given as  \mbox{$C = (W,Z,b^1 * d^1_0, b^2 * d_1^2, x, s, \alpha_0, \alpha_1)$,} with $\alpha_0$ and $\alpha_1$ defined as the compositions
$$\alpha_0: s * b^1 * d^1_0 \Xr{\nu^1 * d^1_0} s^1 * d^1_0 \Xr{\alpha^1_0} x, \quad \quad 
\alpha_1: s * b^2 * d_1^2 \Xr{\nu^2 * d_1^2} 
s^2 * d_1^2 \Xr{\alpha^2_1} x 
$$
Then $H$ is given by $H = (C,h,\eta,\eps)$, with
$\eta$ and $\eps$ defined as the compositions
$$\eta: f_1 \Mr{\eta^1} h^1 \ast d^1_0 \Xr{\gamma^1 \ast d^1_0} h \ast b^1 \ast d^1_0, \quad \quad \eps: h \ast b^2 \ast d_1^2 \Xr{\gamma^2 \ast d_1^2} h^2 \ast d_1^2 \Mr{\eps^2} f_3.$$
\end{lemma}

\begin{proof}
We have to show that, for every 2-functor $(\C,\cc{W}) \mr{F} (\cc{D},\Theta_q)$, $\widehat{FH} = \widehat{FH^2} \ \widehat{FH^1}$. 
Let $\widehat{FC}$, $\widehat{FC^1}$, $\widehat{FC^2}$ be the 2-cells considered in Remark \ref{def:FyHinducen2cell} for $H$, $H^1$, $H^2$. 
We begin by showing {\bf ($\triangle$)} $\widehat{FC} = (Fb^2 \ast  \widehat{FC^2}) \circ F(\delta) \circ (Fb^1 \ast  \widehat{FC^1})$.
By the definition of $\widehat{FC}$, it suffices to show that $Fs \ast ((Fb^2 \ast  \widehat{FC^2}) \circ F(\delta) \circ (Fb^1 \ast  \widehat{FC^1}) = F\alpha$. We compute:

$$
\vcenter{\xymatrix@C=0pc@R=1.5pc{ Fs \did && Fb^1 \did && Fd^1_0 \dcellbbis{\widehat{FC^1}} \\
Fs \did && Fb^1 \dl & \dc{F\delta} & \dr Fd_1^1 \\
Fs \did && Fb^2 \did && Fd^2_0 \dcellbbis{\widehat{FC^2}} \\
Fs && Fb^2 && Fd_1^2
}}
\vcenter{\xymatrix@R=0pc{{\bf 2.} \\ = }}
\vcenter{\xymatrix@C=0pc@R=1.2pc{ Fs \did && Fb^1 \did && Fd^1_0 \dcellbbis{\widehat{FC^1}} \\
Fs && Fb^1 && Fd_1^1 \did \\
& Fs^1 \cl{F\nu^1} &&& Fd_1^1 \\
&& Fx \ar@{-}[ul] \ar@{-}[urr] \ar@{}[u]|>>>>{\ \, F(\alpha^1_1)} \ar@{-}[dl] \ar@{-}[drr] \ar@{}[d]|>>>>{\ F(\alpha^2_0)^{-1}} \\
& Fs^2 \op{F\nu^2} &&& Fd^2_0 \did \\
Fs \did && Fb^2 \did && Fd^2_0 \dcellbbis{\widehat{FC^2}} \\
Fs && Fb^2 && Fd_1^2
}}  
\vcenter{\xymatrix@R=0pc{\eqref{ascensor}  \\ = }} \!\!\!\!\!\!
\vcenter{\xymatrix@C=0.1pc@R=1.2pc{ Fs && Fb^1 && Fd^1_0 \did \\
& Fs^1 \cl{F\nu^1} \did &&& Fd^1 \dcellbbis{\widehat{FC^1}} \\
& Fs^1 &&& Fd_1^1 \\
&& Fx \ar@{-}[ul] \ar@{-}[urr] \ar@{}[u]|>>>>{\ \, F(\alpha^1_1)} \ar@{-}[dl] \ar@{-}[drr] \ar@{}[d]|>>>>{\ F(\alpha^2_0)^{-1}} \\
& Fs^2 \did &&& Fd^2_0 \dcellbbis{\widehat{FC^2}} \\
& Fs^2 \op{F\nu^2} &&& Fd_1^2 \did  \\
Fs && Fb^2 && Fd_1^2
}}
\vcenter{\xymatrix@R=0pc{\ref{def:FyHinducen2cell}  \\ = }}
\!\!\!
\vcenter{\xymatrix@C=-.3pc@R=1.5pc{ Fs && Fb^1 && Fd^1_0 \did \\
& Fs^1 \cl{F\nu^1} &&& Fd^1_0 \\
&& Fx \ar@{-}[ul] \ar@{-}[urr] \ar@{}[u]|>>>>{\ F(\alpha^1_0)} \ar@<-1ex>@{-}[dl] \ar@<1ex>@{-}[drr] \ar@{}[d]|>>>>{\ F(\alpha^2_1)^{-1}} \\
& Fs^2 \op{F\nu^2} &&& Fd_1^2 \did  \\
Fs && Fb^2 && Fd_1^2
}}
$$

\noindent which equals $F\alpha$ by definition. 
It remains to show that 
$$F\eps \circ (Fh \ast  \widehat{FC}) \circ F\eta = F\eps^2 \circ (Fh^2 \ast  \widehat{FC^2}) \circ F\eta^2 \circ F\eps^1 \circ  (Fh^1 \ast  \widehat{FC^1}) \circ F\eta^1.$$
 Clearly by the definitions of $\eps$ and $\eta$ it suffices to show that 
$$F(\gamma^2 \ast  d_1^2) \circ (Fh \ast  \widehat{FC}) \circ F(\gamma^1*  d^1_0) = (Fh^2 \ast  \widehat{FC^2}) \circ F\eta^2 \circ F\eps^1 \circ (Fh^1 \ast  \widehat{FC^1})$$
We compute as follows:


\noindent
$
\vcenter{\xymatrix@C=-0.5pc@R=1.5pc{  & Fh^1 \op{F\gamma^1} &&& Fd^1_0 \did \\
Fh \did && Fb^1 \dl & \dc{\widehat{FC}} & \dr Fd^1_0 \\
Fh && Fb^2 && Fd_1^2 \did \\
& Fh^2 \cl{F\gamma^2} &&& Fd_1^2 }}
\vcenter{\xymatrix@R=0pc{{\bf (\triangle)} \\ = }}
\vcenter{\xymatrix@C=-0.3pc@R=1.5pc{ 
& Fh^1 \op{F\gamma^1} &&& Fd^1_0 \did \\
Fh \did && Fb^1 \did && Fd^1_0 \dcellbbis{\widehat{FC^1}} \\
Fh \did&& Fb^1 \dl & \dc{F\delta} & \dr Fd_1^1 \\
Fh \did && Fb^2 \did && Fd^2 \dcellbbis{\widehat{FC^2}} \\
Fh && Fb^2 && Fd_1^2 \did \\
& Fh^2 \cl{F\gamma^2} &&& Fd_1^2 
}}
\vcenter{\xymatrix@R=0pc{{\bf 1.} \\ = }}
\vcenter{\xymatrix@C=-0.3pc@R=1.2pc{ 
& Fh^1 \op{F\gamma^1} &&& Fd^1_0 \did \\
Fh \did && Fb^1 \did && Fd^1_0 \dcellbbis{\widehat{FC^1}} \\
Fh && Fb^1 && Fd_1^1 \did \\
& Fh^1 \cl{F\gamma^1} &&& Fd_1^1 \\
&& Ff_2 \ar@{-}[ul] \ar@{-}[urr] \ar@{}[u]|>>>>{\ \, F\eps^1} \ar@{-}[dl] \ar@{-}[drr] \ar@{}[d]|>>>>{\ F\eta^2} \\
& Fh^2 \op{F\gamma^2} &&& Fd^2_0 \did \\
Fh \did && Fb^2 \did && Fd^2_0 \dcellbbis{\widehat{FC^2}} \\
Fh && Fb^2 && Fd_1^2 \did \\
& Fh^2 \cl{F\gamma^2} &&& Fd_1^2 
}}
\vcenter{\xymatrix@R=0pc{\eqref{ascensor}  \\ = }}
\vcenter{\xymatrix@C=-0.3pc{Fh^1 \did && Fd^1_0 \dcellbbis{\widehat{FC^1}} \\
Fh^1 && Fd_1^1 \\
& Ff_2 \cl{F\eps^1} \op{F\eta^2} \\
Fh^2 \did && Fd^2_0 \dcellbbis{\widehat{FC^2}} \\
Fh^2 && Fd_1^2}}
$

\vspace{-.5cm}

\end{proof}

\end{appendix}

\end{document}